\documentclass[final]{siamltex}

\usepackage[colorlinks=true,urlcolor=blue]{hyperref}
\usepackage{amstext}
\usepackage{amssymb}
\usepackage{amsmath}
\usepackage{mathrsfs}
\usepackage{array}
\usepackage{enumerate}
\usepackage{verbatim}
\usepackage{url}

\usepackage[]{graphicx}               
\usepackage{color}
\usepackage{appendix}

\newcommand{\norm}[1]{\lVert#1\rVert}
\def\div{\mathrm{div}}
\def\curl{\mathrm{curl}}
\def\vecv{\mathbf{v}}
\def\vecn{\mathbf{n}}
\def\vecx{\mathbf{x}}
\def\vece{\mathbf{e}}
\def\calT{\mathcal{T}}
\def\hdiv{H(\div;\Omega)}

\newtheorem{remark}{Remark}[section]

\title{Genetic Exponentially Fitted Method for Solving Multi-dimensional Drift-diffusion Equations}

\author{Melissa R. Swager\thanks{Department of Mathematics,
        Colorado State University, Fort Collins, Colorado, 80523-1874
        ({\tt swager@math.colostate.edu}).}
        \and Y.~C.~Zhou\thanks{Department of Mathematics,
        Colorado State University, Fort Collins, Colorado, 80523-1874
        ({\tt yzhou@math.colostate.edu}).}}

\begin{document}

\maketitle

\begin{abstract}
A general approach was proposed in this article to develop high-order exponentially fitted basis functions for
finite element approximations of multi-dimensional drift-diffusion equations for modeling biomolecular electrodiffusion
processes. Such methods are highly desirable for achieving numerical stability and efficiency. We found that by utilizing 
the one-one correspondence between continuous piecewise polynomial space of degree $k+1$ and the divergence-free vector 
space of degree $k$, one can construct high-order 2-D exponentially fitted basis functions that are strictly interpolative 
at a selected node set but are discontinuous on edges in general, spanning nonconforming finite element spaces. 
First order convergence was proved for the methods constructed from divergence-free Raviart-Thomas space $RT_0^0$ at 
two different node sets. 
\end{abstract}


\pagestyle{myheadings}
\thispagestyle{plain}
\markboth{MELISSA R. SWAGER AND Y. C. ZHOU}{GENETIC EXPONENTIALLY FITTED METHODS}

\section{Introduction}
We will propose in this article a general approach to construct exponentially fitted methods for numerically solving the drift-diffusion
equation
\begin{equation} \label{eqn:driftdiffusioneq}
- \nabla \cdot (D \nabla u + D \beta u \nabla \phi) = f,
\end{equation}
where $u$ is the density of charged particles, $D$ is the diffusion coefficient, $\beta$ is a problem-dependent constant assumed positive 
in this study, $\phi$ is the given electrostatic potential field. The source function $f$ models the generation or recombination of
the particles, and is assumed here to be independent of $u$. The drift-diffusion equation is widely applied in semiconductor device
modeling, where $u(x)$ is the density of charge carriers \cite{JeromeSemiconductorbook},
and in biomolecular simulations, where $u(x)$ can be the density of ions, ligands, lipids, or other diffusive
molecules \cite{Lu07c,ZhouY2010a,ZhouY2012b}. Generalizations of Eq.(\ref{eqn:driftdiffusioneq}) to model quantum effects,
finite particle sizes, particle-particle correlations and interactions have been made recently \cite{deFalcoC2009a,ZhouY2011a,ZhouY2012b},
in many cases by replacing the electrostatic potential with an effective potential that models these non-electric interactions.
The drift-diffusion equations arising in biomolecular simulations are in general multidimensional due to the intrinsic 3-D structure
of macromolecules. For drift diffusion in bulk solution or through ion channels 3-D drift-diffusion equations are usually adopted so
the charge density can be solved at sufficient temporal and spatial accuracy. For lateral transport of charged particles on
surfaces, such as lateral motion of lipids, lipid clusters, or proteins on membrane surfaces, 2-D drift-diffusion will be very useful 
for the reduction of dimensionality \cite{KiselevV2011a}. Surface drift-diffusion equation rather than 2-D drift-diffusion equation has 
to be used for lateral diffusion on smoothly curved surfaces \cite{AmatoreC2009a,AmatoreC2009b,ZhouY2012b}, for which the gradient and 
divergence operators in (\ref{eqn:driftdiffusioneq}) have to be replaced by surface gradient and surface divergence operators, respectively.
This surface drift-diffusion equation can be solved using surface finite element methods, which share many properties with 2-D finite element methods. 
It is known that numerical solutions of drift-diffusion equations may suffer instability in the advection-dominated regimes that usually appear near the 
strongly charged molecular surfaces such as ion channels, deoxyribonucleic acid (DNA), or membrane surfaces with proteins bound or in the vicinity. 
Numerical solutions of drift-diffusion equation in these biomolecular systems necessitates stable and efficient numerical methods.

Exponentially-fitted methods, known originally as the Scharfetter-Gummel method \cite{ScharfetterD1969},
is a class of finite element methods for solving Eq.(\ref{eqn:driftdiffusioneq}) through an exponential approximation to
the density function $u(x)$ on individual mesh elements. This is made possible by the introduction of the Slotboom
variable
\begin{equation} \label{eqn:slotboom}
\rho = u e^{\beta \phi}, \end{equation} with which Eq.(\ref{eqn:driftdiffusioneq}) can be symmetrized to be
\begin{equation} \label{eqn:driftdiffusioneq_sym}
- \nabla \cdot (De^{-\beta \phi} \nabla \rho) = f.
\end{equation}
The density function $u(x)$ can then be exponentially fitted through locally constant or linear approximations of the
density current $J=De^{-\beta \phi} \nabla \rho$. With exponentially fitted methods the gradient of the
electric potential field can be incorporated to the basis functions, and this makes the methods intrinsically stable
against strong drift when the electric field is strong. This desirable feature has motivated many studies of the
exponentially fitted methods
\cite{AllenD1955,ScharfetterD1969,BrezziF1989a,GattiE1998,SaccoR1998a,HillariretC2004a,PinnauR2004a,AngermannL2003a}.

Exponential fitting for 1-D drift-diffusion equation is obtained by exactly solving a two-point boundary value problem
locally in a divergence-free space of Lagrange basis functions \cite{ScharfetterD1969}. The resultant nodal basis functions
for $u(x)$ involve Bernoulli functions. In multi-dimensional space, however, it is generally difficult to find a divergence-free
space where the basis functions are interpolative at nodes and continuous on edges, for the local problem does not admit an
analytical solution on triangles, quadrilaterals, or tetrahedra. To circumvent this difficulty a number of averaging schemes,
such as exact average, volume harmonic average, and edge harmonic average, were introduced to the exponentially transformed diffusion
coefficient in the symmetrized equation (\ref{eqn:driftdiffusioneq_sym}).
2-D exponentially fitted methods of this type were constructed in \cite{BrezziF1989a}. Such a construction of multi-dimensional exponential
fitting methods is further analyzed in \cite{GattiE1998}, where it is found that the classical Scharfetter-Gummel exponentially fitted
scheme can be exactly reproduced by the harmonic averaging of the piecewise linear Galerkin method. Linear convergence in the exponentially
weighted norm is proved as well. While the convergence and stability of these multi-dimensional methods in numerical computations were not
reported in all of these studies, it was shown that they can well resolve the sharp gradient of the charge density without incurring
any unphysical oscillations. Averaging techniques, such as exact average and volume harmonic average, may suffer from overflow or
lead to over-smeared solutions. Some of the best results are obtained using edge harmonic average, which utilizes approximated exponential
fitting along edges of triangles.

There are continual efforts to obtain genetic exponential fitting in multi-dimensional spaces rather than through averaging. These
include approximate spline fitting \cite{MillerJ1994a,RiordanE1991a,SaccoR1998a,AngermannL2003a,AngermannL2003b}, nonconforming finite
element \cite{SaccoR1997a}, finite volume methods \cite{LazarovR1996a}, and exponentially fitted difference methods \cite{HegartyA1993a}.
The approximate spline methods strongly resemble the 1-D exponential fitting as the locally divergence free current field solutions 
are sought. Such solutions, however, are not unique, and thus approximations will be in place to achieve uniqueness. Nevertheless, locally 
constant approximation to the density current in 2-D violates the nodal interpolation constraints \cite{SaccoR1995a}. This major drawback is overcome
by using locally linear approximation to the density current, and nodal interpolative basis functions for density functions are obtained \cite{SaccoR1999a}.
These basis functions are discontinuous on edges in general so the finite element space is nonconforming. Extension of this
approach to 3-D equations seems feasible. The work in \cite{AngermannL2003a,AngermannL2003b} are among the few attempts to establishing 3-D exponential fitting methods,
where 1-D exponential fitting is sought along individual edges connecting to the vertices, and the nodal basis functions are defined to be the
linear combinations of local 1-D exponentially fitted basis functions. The resulting finite element space is conforming. Despite these efforts
there are very few applications of exponentially fitted methods to realistic multi-dimensional drift-diffusion equations. 
The analysis and application in \cite{PinnauR2004a}, for example, is on 1-D drift-diffusion equation with quantum-corrected
potential. Sometimes Streamline-Upwind/Petrov-Galerkin (SUPG) methods that were developed for advection-diffusion equations have to be
adopted \cite{ChaudhryJ2011b} for stabilization. Moreover, it seems unclear whether the current methodologies can be extended for
constructing high order exponentially fitting methods.

We are motivated to propose in this article a general approach for constructing exponential fitting methods in multi-dimensional spaces.
We start with the divergence-free vector space that is obtained by taking the curl of the nodal basis of $k^{th}$ order Raviart-Thomas
conforming space $P_{k+1}$, to define a prototype of the basis functions for the Slotboom variable $\rho$. The final form of the exponentially fitted
basis functions for $\rho(x)$ and $u(x)$ will be computed by enforcing the nodal interpolative constraints. Our approach is different from that
in \cite{SaccoR1999a} as the latter seeks a divergence-free approximation of the density current and restrictive assumptions are introduced
to get analytical representations of the basis functions. We do not enforce the density current to be divergence-free, and this relaxation
enables our approach to reproduce the standard Lagrange spaces for solving the Poisson equation at the limit of vanishing electric potential
in the drift-diffusion equation. With our approach 2-D exponentially fitted methods of arbitrary high order can be readily constructed by
starting with high-order Raviart-Thomas conforming spaces. We will show that our method does not entail some of the crucial pitfalls
indicated in previous studies. First-order 3-D exponentially fitted method can be established similarly, but the extension of our approach
to high-order methods for 3-D problems is difficult due to the lack of one-one correspondence between $P_{k+1}$ and
the divergence-free vector space in 3-D.

The rest of the paper is organized as follows. In the next section we review some of the most important exponential
fitting methods and summarize their major features. In Section \ref{sec:method} we define the divergence-free spaces and their
basis functions, and describe the procedure to compute exponentially fitted charge density $u(x)$ using these basis functions.
Here we shall show that our construction is general, can reproduce 2-D exponentially fitted method previously developed,
and degenerates to the standard conforming Lagrange finite element spaces. In Section \ref{sec:converg} first order convergence
is proved for constructions based on $RT_0^0$. We make conclusion remarks and outline the future work in Section \ref{sec:summary}.


\section{Genetic Multi-dimensional Exponential Fitting} \label{sec:method}

We first introduce notations that will be used in the discussion below. Let $\Omega \in \mathbb{R}^2$
be a bounded domain with Lipschitz continuous boundary. We introduce for $s>0$
the standard Sobolev space $H^s(\Omega)$ to the functions in $\Omega$, and use the standard inner product
$(\cdot, \cdot)_s$, the norm $\norm{\cdot}_s$, and the semi-norm $|\cdot|_s$. The inner product and
the norm in $L^2(\Omega)$ are denoted by $\| \cdot \|$ and $(\cdot,\cdot)$. The subspace of $L^2(\Omega)$ that consists
of functions with zero mean value is denoted by $L^2_0(\Omega)$. With these we define the space $\hdiv$:
$$ H(\div,\Omega):= \left \{ \mathbf{v}: \mathbf{v} \in (L^2(\Omega))^2; \nabla \cdot \vecv \in L^2(\Omega) \right \}$$
equipped with the norm
$$ \norm{\vecv}_{\hdiv} = \left( \norm{\vecv}^2 + \norm{\nabla \cdot \vecv}^2 \right)^{1/2}.$$

On a regular triangulation $\calT_h$ of $\Omega$ we define the finite element space $V_h$
$$ V_h = \{ \vecv \in \hdiv: \vecv|_K \in V_k(K) ~ \forall K \in \calT_h;  ~\vecv \cdot \vecn |_{\partial \Omega} = 0 \},$$
where $\vecn$ is the outer normal direction on the boundary $\partial \Omega$, and $V_k$ is the space of
vector-valued polynomials of degree $k$ on the element $K$. To fulfill the condition of the vanishing normal component on the boundary $\partial \Omega$
one usually chooses $V_k$ to be the classical Raviart-Thomas element of order $k$ ($RT_k$) \cite{RaviartThomasElements} or the Brezzi-Douglas-Marini
element of order $k$ ($BDM_k$) \cite{BDMElements}.
For other $\hdiv$ spaces constructed more recently the readers may refer to \cite{WangJ2007a,ScheichlR2002}.

We will work on the divergence-free subspace $D_h$ of $V_h$, which is defined to be
$$ D_h = \{ \vecv \in V_h: (\nabla \cdot \vecv, q) = 0~ \forall ~ q \in L^2_0(\Omega) \}.$$
It is known \cite{WangJ2007a,mixedandhybridFEM} that if $V_k $ is chosen to be $RT_k$ then
$$\nabla \cdot V_h = \{ q \in L^2_0(\Omega): q|_K = P_k(K) \}$$
or $BDM_k$ then
$$\textcolor{red}{\nabla \cdot V_h = \{ q \in L^2_0(\Omega): q|_K = P_{k-1}(K) \}}$$
where $P_k$ is the space of polynomials of degree $\le k$, suggesting that $D_h$ is indeed exactly
divergence free:
$$ D_h = \{ \vecv \in V_h: \nabla \cdot \vecv = 0 \}.$$
In case that $V_h = RT_k$, this subspace is also denoted $RT_K^0$.

According to the Helmholtz decomposition, any divergence-free vector $\mathbf{v}$ can be given as the curl of a
potential function $\psi$:
\begin{equation} \label{eqn:curl2D}
\mathbf{v} = \curl \psi = (-\psi_y, \psi_x),
\end{equation}
where the subscripts denote the partial derivatives. Eq.(\ref{eqn:curl2D}) suggests that we can construct a divergence-free subspace
by taking the curl of some appropriate space. This is actually possible, as stated by the following well-known result concerning
the $RT_k$ and $BDM_k$ triangular elements \cite{mixedandhybridFEM,WangJ2009a}:
\begin{theorem}
There exists a one-to-one map curl: $S_h \rightarrow D_h$ where the space
$$ S_h = \{ \psi \in H^1_0(\Omega): \psi_K = P_{k+1}(K), ~ K \in \calT_h \}$$
\textcolor{red}{for triangular $RT_k$ elements with $k \ge 0$ or $BDM_k$ elements with $k \ge 1$}. The dimension of
$D_h = RT_k^0$ is equal to
\begin{equation} \label{eqn:dimensionPKRK}
\mathrm{dim} P_{k+1}(K) - 1 = \frac{1}{2}(k+1)(k+4).
\end{equation}
\end{theorem}
Similar correspondence between $H(\mathrm{curl})$ spaces and the divergence free subspaces $D_h$ in 3-D is
also indicated \cite{DouglasA2000a}. Two sample spaces of $D_h=RT_k^0$ on the reference triangle are given here:
\begin{eqnarray}
RT_0^0 & = & \mathrm{span} \{ (1,0)^T,(0,1)^T \}, \label{eqn:RT00} \\
RT_1^0 & = & \mathrm{span} \{ (x, 1-2x - y)^T, (0, -1+4x)^T, (-x, y)^T, (1-4y,0)^T, (-1+x+2y,-y)^T \}. \label{eqn:RT10}
\end{eqnarray}

We are now in a position to construct the exponentially fitted methods for solving the following boundary
value problem for Eq.(\ref{eqn:driftdiffusioneq_sym}): \textit{Find $\rho \in H^1(\Omega)$ such that}
\begin{equation} \label{eqn:problem}
\begin{split}
 -\nabla \cdot J(\rho) = f  & \qquad \mbox{in} ~ \Omega, \\
 J(\rho) = D e^{-\beta \phi} \nabla \rho = D (\nabla u + \beta u \nabla \phi) & \\
 \rho = g e^{\beta \phi} & \qquad \mbox{on} ~\Gamma_D \subset \partial \Omega, \\
 J(\rho) \cdot \vecn = 0 & \qquad \mbox{on} ~ \Gamma_N = \partial \Omega \setminus \Gamma_D,
\end{split}
\end{equation}
where $\Gamma_D$ and $\Gamma_N$ are the Dirichlet and Neumann subsets of the boundary $\partial \Omega$,
respectively. We denote by $u_h$ and $J_h(u_h) = D(\nabla u_h + \beta u_h \nabla \phi)$ the finite element
approximations of $u$ and $J(u)$. Let $\{ \vecv_i \}$ be the basis of $D_h$ on a triangular element $K$,
then the approximate density current function $J_h(u_h)$ in $K$ can be given as a linear combination of
$\{ \vecv_i \}$:
\begin{equation} \label{eqn:Jexpansion}
J_h(u_h)|_K = \sum_{i=1}^{N_k} c_i \vecv_i,
\end{equation}
where $N_k$ is the dimension of $D_h|_K$. Notice that $J_h(u_h)$ can be given in terms of the Slotboom
variable $\rho_h = u_h e^{\beta \phi}$ as $J_h(\rho_h)= D e^{-\beta \phi} \nabla \rho_h$, we shall have
\begin{equation} \label{eqn:rhoexpansion}
D e^{-\beta \phi} \nabla \rho_h |_K = \sum_{i=1}^{N_k} c_i \vecv_i.
\end{equation}
It remains to find the basis functions $\rho_j$ of the certain finite element space on $K$ in which $\rho_h$ can be
approximated.

\subsection{Attempts to approximate density current exactly in divergence-free spaces} \label{subsect:attempts}
If each basis function $\rho_j$ has the representation
\begin{equation} \label{eqn:rhoexpansion_2}
\nabla \rho_j = \frac{1}{D} e^{\beta \phi} \sum_{i=1}^{N_k} m_{j,i} \vecv_i,
\end{equation}
for some set of real numbers $\{ m_{j,i} \}$, then (i) $\nabla \rho_j$ and $\nabla \rho_h$ will be exactly divergence-free, and (ii)
the right-hand side of Eq.(\ref{eqn:rhoexpansion_2}), as the gradient of a scalar, is curl free.
Consequently, its integration along any curve $l(P_1,P_2)$ connecting two given points in $K$ is path-independent, for
\begin{equation} \label{eqn:path_independent}
\int_{l(P_1,P_2)} \nabla \rho_j \cdot dl = \rho_j(P_2) - \rho_j(P_1).
\end{equation}
For this reason one can choose the path to be the line connecting $(x,y)$ and the starting point $(x_0,y_0)$, c.f. Figure \ref{fig:intgpath}.
Then the basis function $\rho_j(x,y)$ can be computed from
\begin{equation} \label{eqn:rhocomputing}
\rho_j(x,y) = \rho_j(x_0,y_0) + \frac{1}{D} \sum_{i=1}^{N_k} m_{j,i} \int_{0}^{S(x,y)}
e^{\beta \phi(x(s),y(s))} \vecv_i(x(s),y(s)) \cdot \vecn ds,
\end{equation}
where $S(x,y) = \sqrt{(x-x_0)^2 + (y - y_0)^2}$ and $\vecn = (x-x_0, y-y_0)/S$ is the directional vector pointing $(x_0,y_0)$
to $(x,y)$. Consequently,
\begin{equation} \label{eqn:ucomputing}
u_j(x,y) = u_j(x_0,y_0) + \frac{e^{-\beta \phi(x,y)}}{D} \sum_{i=1}^{N_k} m_{j,i}
\int_{0}^{S(x,y)} e^{\beta \phi(x(s),y(s))} \vecv_i(x(s),y(s)) \cdot \vecn ds,
\end{equation}
with \textcolor{red}{$u_j(x_0,y_0) = \rho_j(x_0,y_0)e^{-\beta \phi(x_0,y_0)}$}. These functions $\{ u_j(x,y) \}$ form a basis function of $u_h$ in the element $K$.
On the standard reference triangle one can choose a two-section path $(x_0,y_0) \rightarrow (x,y_0) \rightarrow (x,y)$, each section
parallel to the axes, giving rise to
\begin{equation} \label{eqn:rhocomputing_2}
\rho_j(x,y) = \rho_j(x_0,y_0) + \frac{1}{D} \sum_{i=1}^{N_k} m_{j,i} \left ( \int_{x_0}^x e^{\beta \phi(s,y_0)} \vecv_i^x(s,y_0) ds +
\int_{y_0}^y e^{\beta \phi(x,t)} \vecv_i^y(x,t) dt \right),
\end{equation}
where $(\vecv_i^x, \vecv_i^y)$ are the two components of the vector $\vecv_i$, and correspondingly
\begin{equation} \label{eqn:ucomputing_2}
u_j(x,y) = u_j(x_0,y_0) + \frac{e^{-\beta \phi}}{D} \sum_{i=1}^{N_k} m_{j,i} \left ( \int_{x_0}^x e^{\beta \phi(s,y_0)} \vecv_i^x(s,y_0) ds
+ \int_{y_0}^y e^{\beta \phi(x,t)} \vecv_i^y(x,t) dt \right).
\end{equation}
\begin{figure}[!ht]
\begin{center}
\includegraphics[width=5cm]{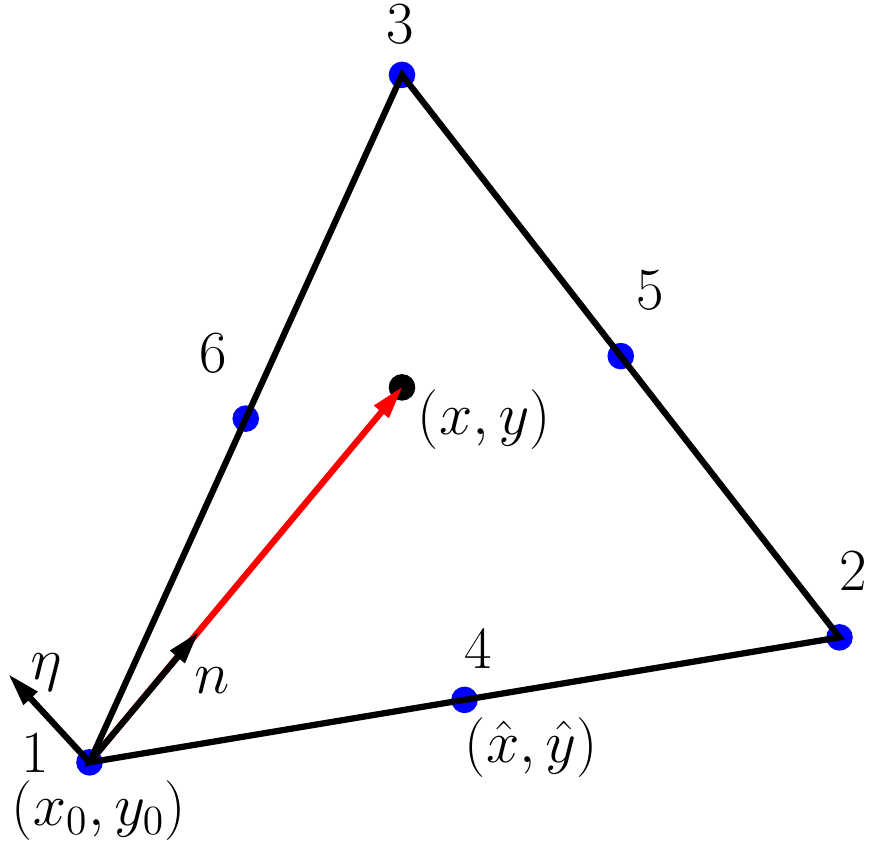} \hspace{1cm}
\caption{One needs to choose the starting point $(x_0,y_0)$ and the path from $(x_0,y_0)$ to $(x,y)$ for
the integrals in Eq.(\ref{eqn:rhocomputing},\ref{eqn:ucomputing}). For $RT_0^0$ one can place $(x_0,y_0)$ at
one vertex and enforce interpolative constraints at all three vertices; or choose one of the edge
centers be the start point and do enforcement at all three edge centers. Different basis functions will
be generated for different paths.
}
\label{fig:intgpath}
\end{center}
\end{figure}

To uniquely determine the total $N_k$ coefficients $m_{ji}$ and the constant $u_j(x_0,y_0)$ in the expression (\ref{eqn:ucomputing})
we will enforce the interpolative constraints at $N_k + 1$ nodes. \textcolor{red}{For example, for $D_h = RT_0^0$ on the triangle in
Figure (\ref{fig:intgpath}) we can expand equation (\ref{eqn:ucomputing_2}) with $j=2$ at three vertices $(1,2,3)$ to get the following
\begin{equation} \label{eqn:ucomputing_ex1}
u_2(x,y) = u_2(x_0,y_0) + \frac{e^{-\beta \phi(x_3,y_3)}}{D} \sum_{i=1}^{N_k} m_{1,i} \left ( \int_{x_0}^{x} e^{\beta \phi(s,y_0)} \vecv_i^x(s,y_0) ds
+ \int_{y_0}^{y} e^{\beta \phi(x,t)} \vecv_i^y(x,t) dt \right).
\end{equation}
}

\textcolor{red}{If we choose $(x_1, y_1)$ to be $(x_0,y_0)$ then we must have from the interpolative condition that
\begin{eqnarray} 
0 & = & u_2(x_1,y_1) = u_2(x_0,y_0) + \frac{e^{-\beta \phi(x_1,y_1)}}{D} \sum_{i=1}^{N_k} m_{2,i} \left ( \int_{x_0}^{x_1} e^{\beta \phi(s,y_0)} \vecv_i^x(s,y_0) ds
+ \int_{y_0}^{y_1} e^{\beta \phi(x_1,t)} \vecv_i^y(x_1,t) dt \right), \label{eqn:ucomputing_ex2} \\
1 & = & u_2(x_2,y_2) = u_2(x_0,y_0)+ \frac{e^{-\beta \phi (x_2,y_2)}}{D} \sum_{i=1}^{N_k} m_{2,i} \left ( \int_{x_0}^{x_2} e^{\beta \phi(s,y_0)} \vecv_i^x(s,y_0) ds
+ \int_{y_0}^{y_2} e^{\beta \phi(x_2,t)} \vecv_i^y(x_2,t) dt \right), \label{eqn:ucomputing_ex3} \\
0 & = & u_2(x_3,y_3) = u_2(x_0,y_0) + \frac{e^{-\beta \phi (x_3,y_3)}}{D} \sum_{i=1}^{N_k} m_{2,i} \left ( \int_{x_0}^{x_3} e^{\beta \phi(s,y_0)} \vecv_i^x(s,y_0) ds
+ \int_{y_0}^{y_3} e^{\beta \phi(x_3,t)} \vecv_i^y(x_3,t) dt \right). \label{eqn:ucomputing_ex4}
\end{eqnarray}
}

\textcolor{red}{These three equations form a linear system
\begin{equation} \label{eqn:linearsystem_m2}
\left (
\begin{array}{rrr}
0 & 0 & 1 \\
F_{11} & F_{12} & 1 \\
F_{21} & F_{22} & 1 
\end{array}
\right)
\left (
\begin{array}{r}
m_{21}  \\
m_{22}  \\
u_2(x_0,y_0)
\end{array}
\right)
=
\left (
\begin{array}{r}
0 \\
1  \\
0
\end{array}
\right),
\end{equation}
where $F_{1i}$ and $F_{2i}$ are defind as
\begin{eqnarray}
 F_{1i} & = & \frac{e^{-\beta \phi (x_2,y_2)}}{D} \int_{x_0}^{x_2} e^{\beta \phi(s,y_0)} \vecv_i^x(s,y_0) ds
+ \int_{y_0}^{y_2} e^{\beta \phi(x_3,t)} \vecv_i^y(x_3,t) dt , \quad i \le 2, \label{eqn:Ffunction2} \\
 F_{2i} & = & \frac{e^{-\beta \phi (x_3,y_3)}}{D} \int_{x_0}^{x_3} e^{\beta \phi(s,y_0)} \vecv_i^x(s,y_0) ds
+ \int_{y_0}^{y_3} e^{\beta \phi(x_3,t)} \vecv_i^y(x_3,t) dt , \quad i \le 2. \label{eqn:Ffunction3}
\end{eqnarray}
}

\textcolor{red}{Therefore, with the similar expansions of the other two basis functions $u_1(x,y), u_3(x,y)$, one can arrive at the general linear system below:
\begin{equation} \label{eqn:linearsystem_m}
\left (
\begin{array}{rrr}
0 & 0 & 1 \\
F_{11} & F_{12} & 1 \\
F_{21} & F_{22} & 1 
\end{array}
\right)
\left (
\begin{array}{rrr}
m_{11} & m_{21} & m_{31} \\
m_{12} & m_{22} & m_{32} \\
u_1(x_0,y_0) & u_2(x_0,y_0) & u_3(x_0,y_0)
\end{array}
\right)
=
\mathrm{I},
\end{equation}
where $\mathrm{I}$ is the identity matrix, and
\begin{equation} \label{eqn:Ffunction}
 F_{ji} = \frac{e^{-\beta \phi(x_j,y_j)}}{D}
\int_{0}^{S(x_j,y_j)} e^{\beta \phi(x(s),y(s))} \vecv_i(x(s),y(s)) \cdot \vecn ds, \quad j \le 2.
\end{equation}
}

The following theorem proves that the matrix $F$ is invertible for either set of the nodes, \textcolor{red}{i.e., the vertices $(1,2,3)$ or the edge centers $(4,5,6)$}.
\begin{theorem} \label{eqn:thm_interpolative}
The matrix $F$ is nonsingular for the basis $\{ \vecv_i \}$ of $RT_0^0$.
\end{theorem}
\begin{proof}
We notice that every component $F_{ji}$ represents a linear transformation of $\vecv_i$, i.e.,
$F_{ji} = \mathcal{F}_{j}(\vecv_i)$ where $\mathcal{F}_{j}: D_h|_K \rightarrow \mathbb{R}$.
Consequently it is sufficient to prove that $\mathcal{F}_j(\vecv)=0$ for $j=1,2$ if and only if $\vecv = 0$.
We only need to prove that if $\mathcal{F}_j(\vecv)=0$ for $j=1,2$ then $\vecv=0$, as the reverse statement
is trivial.  We denote $\vecv$ by $\vecv = (a,b)$ because any $\vecv \in D_h|_K$ is a constant vector. If
the vertices are chosen to be the node set for enforcing the interpolative constraints, we have
\begin{eqnarray} \label{eqn:F_nonsingular}
\mathcal{F}_1(\vecv) & = & \frac{e^{-\beta \phi(x_2,y_2)}}{D} \int_{0}^{S(x_2,y_2)} e^{\beta \phi(x(s),y(s))} (a n_2 - b n_1) ds =
p(a n_2 - b n_1), \label{eqn:F_nonsingular_1} \\
\mathcal{F}_2(\vecv) & = & \frac{e^{-\beta \phi(x_3,y_3)}}{D} \int_{0}^{S(x_3,y_3)} e^{\beta \phi(x(s),y(s))} (a n_4 - b n_3) ds =
\textcolor{red}{q(a n_4 - b n_3)}, \label{eqn:F_nonsingular_2}
\end{eqnarray}
where $(n_1,n_2)$ and $(n_3,n_4)$ are the unit vectors in the direction $(x_1,y_1) \rightarrow (x_2,y_2)$ and
$(x_1,y_1) \rightarrow (x_3,y_3)$, respectively, and $p,q$ are constants. It following immediately that if
Eqs.(\ref{eqn:F_nonsingular_1},\ref{eqn:F_nonsingular_2}) are both equal to zero then $a=b=0$ since the vectors
$(n_1,n_2)$ and $(n_3,n_4)$ are not parallel. Similar result holds true if the edge nodes are chosen for 
enforcing the interpolative constraints, with $(\hat{x},\hat{y})$ being the start point of the integral path.
\end{proof}

It turned out that parameters $\{ m_{ji} \}$ determined as such do not fit the expansion (\ref{eqn:rhoexpansion_2}). Indeed,
if the curl-free nature of $\displaystyle{\sum_{i=1}^{N_k} m_{ji} \vecv_i }$ is enforced along with the interpolative
constraints we will have an over-determined problem for $m_{ji}$, suggesting that one can not find a set of interpolative
basis functions for $u_h$ that constitute a constant density current. Since this curl-free condition is not satisfied,
different basis functions $\rho_j, u_j$ can be generated if different paths are chosen for the integrations.
Figure (\ref{fig:difference}) shows the difference between the two basis functions that are obtained along different
integral paths. The differences on the edges suggest the basis functions on neighboring triangles are discontinuous in general. 
The finite element space is hence nonconforming. In addition, this method can not be generalized to construct high order exponentially 
fitted method. For example, direct computation on $D_h = RT_1^0$ in the reference triangle with $\phi = 0$, along the integration 
path $(0,0) \rightarrow (x,0) \rightarrow (x,y)$, gives rise to a singular matrix $F$:
\begin{equation*}
\left (
\begin{array}{rrrrrr}
 1/8  &  0    &  -1/8  &  1/2 &  -3/8 & 1 \\
 1/2  &  0    &  -1/2  &  1   &  -1/2 & 1 \\
 1/4  &  0    &     0  &  0   &  -1/4 & 1 \\
 1/2  & -1    &   1/2  &  0   &  -1/2 & 1 \\
 3/8  & -1/2  &   1/8  &  0   &  -1/8 & 1 \\
0 & 0 & 0 & 0 & 0 & 1
\end{array}
\right).
\end{equation*}
\begin{figure}[!ht]
\begin{center}
\includegraphics[width=5cm]{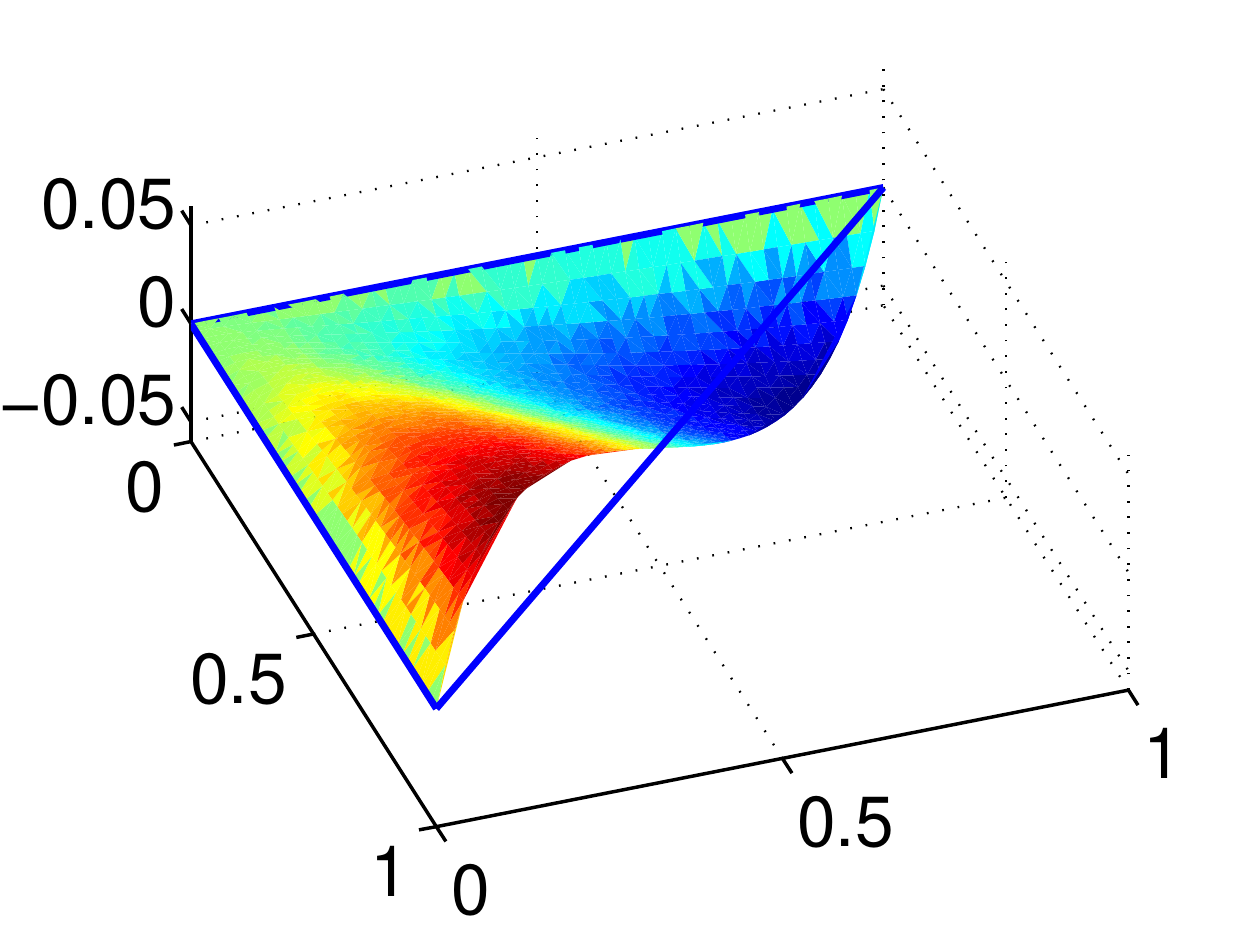} \hspace{1cm}
\includegraphics[width=5cm]{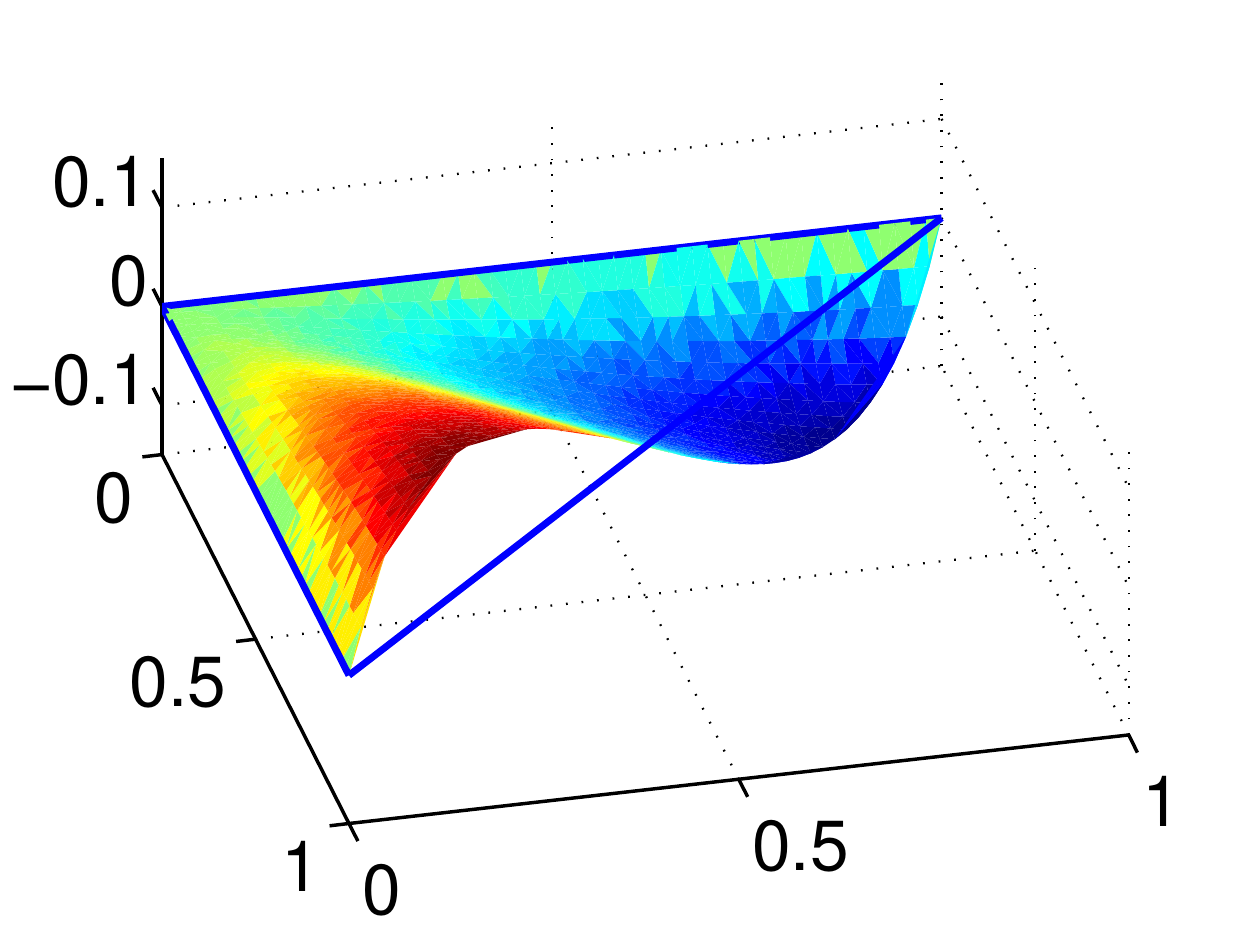}
\caption{Different basis functions $u_j$ can be obtained by choosing different integral path $s$. Left: Difference of $u_1$ in
Eq.(\ref{eqn:ucomputing}) computed along $s$ being $(0,0) \rightarrow (x,0) \rightarrow (x,y)$ and $(0,0) \rightarrow (0,y) \rightarrow (x,y)$.
Right: Difference of $u_1$ computed through (\ref{eqn:rhocomputing_new}) and (\ref{eqn:rhocomputing_new_2}).
}
\label{fig:difference}
\end{center}
\end{figure}

\begin{remark}
A divergence-free finite element space
$$\mathbb{DF}(K) = \mathrm{span} \{ 1, e^{\nabla \phi \cdot \vecx}, (\nabla \phi \times \vecx) \cdot \vece_3 \}$$
is constructed in \cite{SaccoR1999a}, by assuming local linear approximations of the electric potential
and the density current $J_h$. The six parameters of the latter and one constant in $u_j$ are determined by applying 
seven conditions, namely one divergence-free condition, three linear conditions arising in the curl-free nature of
$\nabla \rho_h$, and three interpolation constraints. The linear approximation of the density current provides more
degrees of freedom than the constant approximation in $RT_0^0$ presented above. These additional degrees of freedom
make it possible to enforce the curl-free condition along with the interpolative constraints.
The component $(\nabla \phi \times \vecx) \cdot e_z $ of space allows the simulation of the diffusion flow orthogonal to
the electric field lines. Interestingly, we can prove that the nonconforming space $\mathbb{DF}(K)$ is a special case of our
construction. \textcolor{red}{Consider the standard reference triangle and let $D_h|_K = RT_0^0$ then $D_h|_K = \mathrm{span}\{ (1,0)^T,(0,1)^T \}$. 
Let Path 1 be defined by $(1,0) \rightarrow (x,0) \rightarrow (x,y)$ and Path 2 be defined by $(1,0) \rightarrow (0,y) \rightarrow (x,y)$.
Note that by averaging the results for the two symmetric paths we will arrive at the final solution $u_j(x,y)$ and if we assume the same 
linear potential in the element $\phi|_K = ax + by + c$, then for Path 1 we have}
\textcolor{red}{
\begin{eqnarray}
u_j^1(x,y) & = & u_j(x_0,y_0) + \frac{e^{-\beta \phi}}{D} \sum_{i=1}^{N_k} m_{j,i} \left ( \int_{x_0}^x e^{\beta \phi(s,y_0)} \vecv_i^x(s,0) ds
+ \int_{y_0}^y e^{\beta \phi(x,t)} \vecv_i^y(x,t) dt \right) \nonumber \\
 & = & u_j(x_0,y_0) + \frac{e^{-\beta \phi}}{D} \left (  m_{j,1} \int_{x_0}^x e^{\beta \phi(s,y_0)} ds + m_{j,2}
 \int_{y_0}^y e^{\beta \phi(x,t)}  dt \right) \nonumber \\
 & = & u_j(x_0,y_0) + \frac{e^{-\beta \phi}}{D}\left( m_{j,1} \left (\frac{1}{\beta a} e^{\beta(ax+c)}-\frac{1}{\beta a} e^{\beta c} \right) +  m_{j,2} \left (\frac{1}{\beta b} e^{\beta(ax+by+c)}-\frac{1}{\beta b} e^{\beta(ax+c)} \right )\right) \nonumber \\
 & = & u_j(x_0,y_0) +  \frac{m_{j,2}}{\beta b D} + e^{-\beta b y} \left( \frac{m_{j,1}}{\beta a D} - \frac{m_{j,2}}{\beta b D}\right) - \frac{m_{j,1}}{\beta a D}e^{-\beta(ax+by)}.
\end{eqnarray}
For Path 2 we have
\begin{eqnarray}
u_j^2(x,y) & = & u_j(x_0,y_0) + \frac{e^{-\beta \phi}}{D} \sum_{i=1}^{N_k} m_{j,i} \left ( \int_{y_0}^y e^{\beta \phi(x_0,s)} \vecv_i^x(0,s) ds
+ \int_{x_0}^x e^{\beta \phi(t,y)} \vecv_i^y(t,y) dt \right) \nonumber \\
 & = & u_j(x_0,y_0) + \frac{e^{-\beta \phi}}{D} \left (  m_{j,1} \int_{y_0}^y e^{\beta \phi(x_0,s)} ds + m_{j,2}
 \int_{x_0}^x e^{\beta \phi(t,y)}  dt \right) \nonumber \\
 & = & u_j(x_0,y_0) + \frac{e^{-\beta \phi}}{D}\left( m_{j,1} \left (\frac{1}{\beta b} e^{\beta(by+c)}-\frac{1}{\beta b} e^{\beta c} \right) +  m_{j,2} \left (\frac{1}{\beta a} e^{\beta(ax+by+c)}-\frac{1}{\beta a} e^{\beta(by+c)} \right )\right) \nonumber \\
 & = & u_j(x_0,y_0) +  \frac{m_{j,2}}{\beta a D} + e^{-\beta ax} \left( \frac{m_{j,1}}{\beta b D} - \frac{m_{j,2}}{\beta a D}\right) - \frac{m_{j,1}}{\beta b D}e^{-\beta(ax+by)}.
\end{eqnarray}
}

\textcolor{red}{Assuming that the integration in Eq. (\ref{eqn:Ffunction}) is path-independent, then the results from Path 1 and Path 2 are the same, and equal to their average
\begin{eqnarray}
u_j(x,y) &=& \frac{u_j^1(x,y)+u_j^2(x,y)}{2} \nonumber \\
         & =&  u_j(x_0,y_0) +  \frac{m_{j,2}}{2\beta D}\left(\frac{1}{a}+ \frac{1}{b} \right) 
+  e^{-\beta b y} \left( \frac{m_{j,1}}{2\beta a D} - \frac{m_{j,2}}{2\beta b D}\right) \nonumber \\
         &+& e^{-\beta ax} \left( \frac{m_{j,1}}{2\beta b D} -\frac{m_{j,2}}{2\beta a D}\right)- m_{j,1} \left ( \frac{1}{2\beta b D}+\frac{1}{2\beta a D} \right)e^{-\beta(ax+by)}
 \label{eqn:Uexpansion}
\end{eqnarray}
Applying the approximation $e^{x} \approx 1 + x$ we arrive at
\begin{eqnarray}
u_j(x,y)  &\approx&  u_j(x_0,y_0) +  \frac{m_{j,2}}{2\beta D}\left(\frac{1}{a}+ \frac{1}{b} \right) + \left( \frac{m_{j,1}}{2\beta a D} - \frac{m_{j,2}}{2\beta b D}\right) 
+ \left( \frac{m_{j,1}}{2\beta b D} - \frac{m_{j,2}}{2\beta a D}\right)  \nonumber \\
 & - & m_{j,1} \left ( \frac{1}{2\beta b D}+\frac{1}{2\beta a D} \right)e^{-\beta(ax+by)}    
          - \beta by \left( \frac{m_{j,1}}{2\beta a D} - \frac{m_{j,2}}{2\beta b D}\right) - \beta a x \left( \frac{m_{j,1}}{2\beta b D} - \frac{m_{j,2}}{2\beta a D}\right) \nonumber \\
 & = & \alpha_j + \eta_j e^{-\beta \nabla \phi \cdot \vecx} + \gamma_j (\nabla \phi \times \vecx) \cdot \vece_3,
 \label{eqn:Uequiv_1}
\end{eqnarray}
for some reorganized constants $\alpha_j, \eta_j, \gamma_j$.} This last
formula (\ref{eqn:Uequiv_1}) is of the exact same form as in \cite{SaccoR1999a}. However, the restrictive assumption $a,b \ne 0$
in the local linear representation of the electric potential is no longer necessary in our construction here because we do not
compute the analytical integration of the function $e^{\beta \phi}$ and therefore the analytical form of the potential $\phi$
is not needed. When $u_j$ of the representation in Eq.(\ref{eqn:Uexpansion}) is used to compute the current $J_h$, the first term
will produce a current component in the direction of the electric field; the second term does not contribute to $J_h$, and the last
two terms will provide additional fluxes that are not necessarily aligned with the electric field. Eq.(\ref{eqn:Uequiv_1}) indicates
that these additional fluxes are either not necessarily orthogonal to the electric field in general. Under very special circumstances,
for instances $a=b$ in Eq.(\ref{eqn:Uexpansion}), one can solve the same $m_{j,1}, m_{j,2}$, and this crosswind diffusion flux will
disappear in Eq.(\ref{eqn:Uexpansion}) as it will be absorbed into the second term.

It is worth noting the divergence-free condition was only approximately enforced to get a constant current field in \cite{SaccoR1995a}.
\end{remark}

\subsection{General construction of exponentially fitted methods by giving up divergence-free condition}

We would note that it is not fully justified to enforce the divergence-free condition in constructing the exponentially fitted finite
element space. Divergence-free condition is seldom enforced in solving the Laplace equation, which is the limit of the drift-diffusion
equation (\ref{eqn:driftdiffusioneq}) at the limit of a vanishing electric field and with $f(x)=0$. Divergence-free condition shall not
be used for solving drift-diffusion equations with inhomogeneous source function $f(x)$, or unsteady drift-diffusion equations.

We will give up the divergence-free condition to seek an alternative construction. Our construction for the drift-diffusion equation
must regenerate the standard finite element space $P_{k+1}$ for solving the Poisson equation at the limit of vanishing electric field.
These standard basis functions $\psi_j$ of $P_{k+1}$ can be derived from the known basis of $D_h$ through the following correspondence
\begin{equation} \label{eqn:poisson_antiderivation}
\left (-\frac{\partial \psi_j}{\partial y}, \frac{\partial \psi_j}{\partial x} \right)^T = \sum_{i=1}^{N_k} m_{ji} \vecv_i
\end{equation}
and the enforcement of the interpolative constraints. We are motivated to construct the basis $\rho_j$ for $\rho_h$ in a
similar manner
\begin{equation} \label{eqn:driftdiffusion_antiderivation}
\left (-\frac{\partial \rho_j}{\partial y}, \frac{\partial \rho_j}{\partial x} \right)^T = \frac{e^{\beta \phi}}{D}
\sum_{i=1}^{N_k} m_{ji} \vecv_i.
\end{equation}
On the reference triangle $\hat{K}$, we can evaluate $\rho_j$ at an arbitrary point $(x,y)$ via different paths
\begin{equation} \label{eqn:rhocomputing_new}
\rho_j(x,y) = \rho_j(x_0,y_0) + \frac{1}{D} \sum_{i=1}^{N_k} m_{j,i} \left ( \int_{x_0}^x e^{\beta \phi(s,y_0)} \vecv_i^y(s,y_0) dt
- \int_{y_0}^y e^{\beta \phi(x,t)} \vecv_i^x(x,t) dt \right),
\end{equation}
or
\begin{equation} \label{eqn:rhocomputing_new_2}
\rho_j(x,y) =  \rho_j(x_0,y_0) + \frac{1}{D} \sum_{i=1}^{N_k} m_{j,i} \left ( \int_{x_0}^x e^{\beta \phi(s,y)} \vecv_i^y(s,y) dt
- \int_{y_0}^y e^{\beta \phi(x_0,t)} \vecv_i^x(x_0,t) dt \right).
\end{equation}
The expansion coefficients $m_{ji}$ and the constants $\rho_j(x_0,y_0)$  can be solved from the following linear system
similar to (\ref{eqn:linearsystem_m})
\begin{equation} \label{eqn:F_standard}
\left (
\begin{array}{rrrr}
\ddots &        &        & 1      \\
       & F_{ij} &        & \vdots \\
       &        & \ddots & 1      \\
 0     & \cdots &  0     & 1
\end{array}
\right )
\left (
\begin{array}{cccc}
\ddots &        &        &  m_{N_k+1,1}     \\
       & m_{ji} &        &   \vdots     \\
       &        & \ddots &  m_{N_k+1,N_k}   \\
 \rho_1(x_0,y_0) & \cdots &  \rho_{N_k}(x_0,y_0)   & \rho_{N_k+1}(x_0,y_0)
\end{array}
\right ) = I_{(N_k+1) \times (N_k+1)}
\end{equation}
resulting from the enforcement of the interpolative constraints at selected node sets, with
\begin{equation} \label{eqn:Fij_new_phi}
F_{ji} = \frac{1}{D} \left ( \int_{x_0}^{x_j} e^{\beta \phi(s,y_0)} \vecv_i^y(s,y_0) dt
- \int_{y_0}^{y_j} e^{\beta \phi(x_j,t)} \vecv_i^x(x_j,t) dt \right) , \quad 1 \le j \le N_k+1, ~ 1 \le i \le N_k,
\end{equation}
We can then evaluate
\begin{equation} \label{eqn:ucomputing_3}
u_j(x,y) = \rho_j(x,y) e^{-\beta \phi(x,y)}.
\end{equation}
Proper scaling can make $u_j(x,y)$ to be interpolative as well.
It turns out in general that $N_k + 1 = \mathrm{dim}P_{k+1}(K)$, meaning the dimension of the finite element space for $u_h|_K$ corresponding to the divergence-free
subspace $D_h|_K = RT_k^0|_K$ is equal to the dimension of the space $S_h|_K=P_{k+1}(K)$ whose curl produces $D_h|_K$. The one-one correspondence between $RT_k^0$
and $P_{k+1}$ indicates that the matrix $F$ corresponding to Eq.(\ref{eqn:poisson_antiderivation}), i.e., $F$ with
\begin{equation} \label{eqn:Fij_new_nophi}
F_{ji} = \int_{x_0}^{x_j} \vecv_i^y(s,y_0) dt - \int_{y_0}^{y_j} \vecv_i^x(x_j,t) dt
\end{equation}
must be invertible. We shall note that the finite element spaces constructed here are nonconforming and depend on the integral path,
similar to the spaces obtained in subsection \ref{subsect:attempts}, because the solved parameters $m_{ji}$ may not fit
Eq.(\ref{eqn:driftdiffusion_antiderivation}). They do fit when $\phi=0$.

Two examples of the node set for the finite element space $u_h|_K=\mathrm{span} \{ u_j \}$ constructed on $D_h| = RT_0^0$ on the
reference triangle are shown in Figure \ref{fig:Uelement_example}. For a weak potential $\phi =  e^{-2\sqrt{x^2 + y^2}}$ the
exponentially fitted basis functions are very similar to the basis of $P_1(\hat{K})$ at the same set of nodes, c.f.
Figure \ref{fig:basisAN} and \ref{fig:basisBN}. The nature of exponential fitting is more clearly illustrated when the electric field
is strong, c.f. Figure \ref{fig:basisANS} and \ref{fig:basisBNS}. It can be seen that the finite element spaces spanned by $\rho_j, u_j$
are invariant with respect to affine linear maps, regardless of the choice of node set. This feature is of particular importance
to the estimation of interpolation errors through mappings to the reference triangles. The basis functions $u_j$ constructed
on $D_h = RT_1^0$ with $\phi = \exp^{-4\sqrt{x^2 + y^2}}$ are shown in Figure \ref{fig:basisO2}, where the characters of
interpolation at nodes and exponential fitting are also well depicted. The method developed here thus provides for the first time a
feasible approach to systematically construct high-order exponentially fitted methods for solving 2-D drift-diffusion equation.

\begin{figure}[!ht]
\begin{center}
\includegraphics[width=5cm]{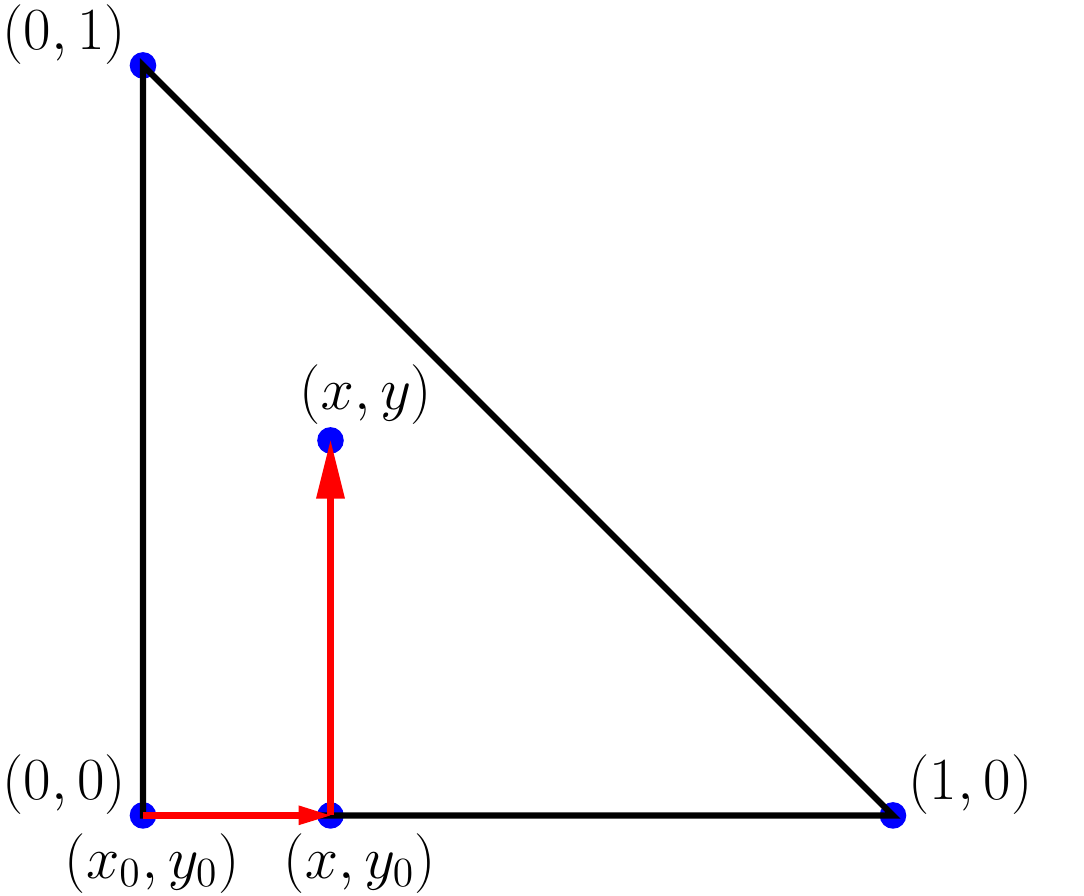} \hspace{1cm}
\includegraphics[width=5cm]{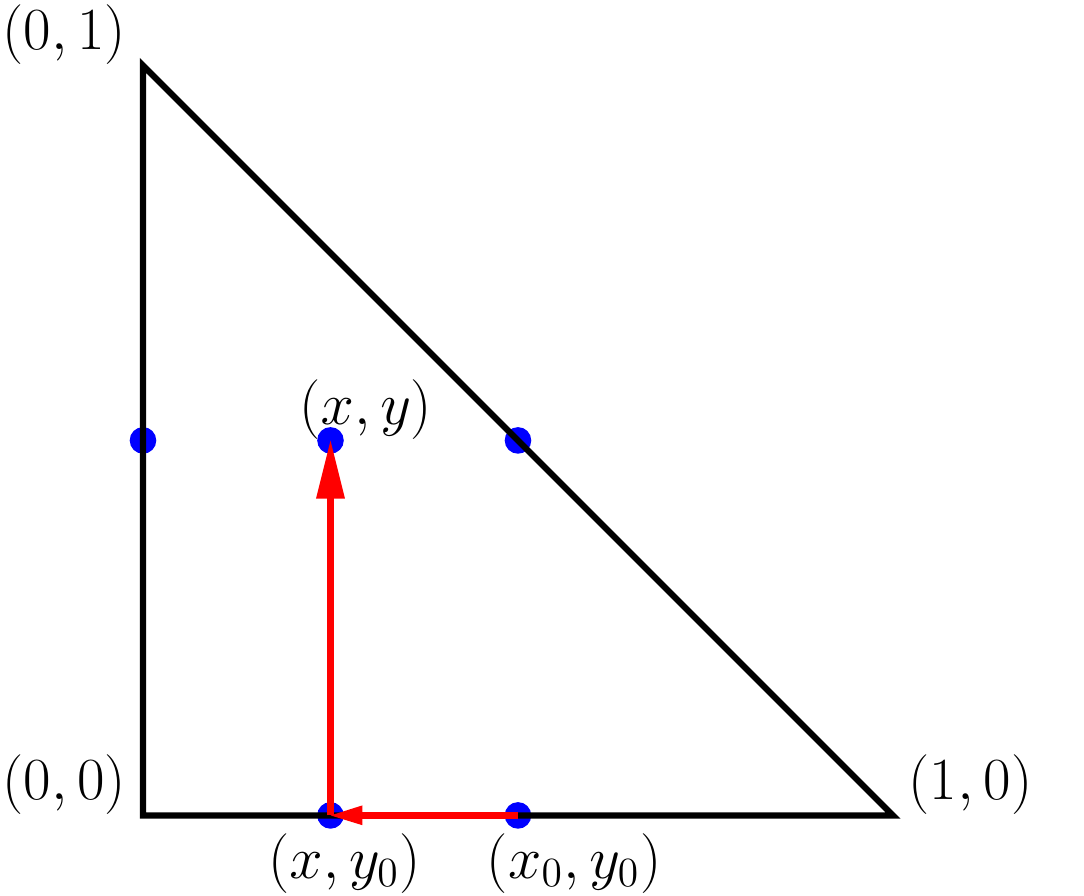}
\caption{Difference node sets are chosen for enforcing the interpolative constraints for $D_h = RT_0^0$, giving rise to different
basis functions of $u_h$. Left: triangle vertices; Right: Edge nodes. The starting point $(x_0,y_0)$ and the path for the integration
vary accordingly.}
\label{fig:Uelement_example}
\end{center}
\end{figure}

\begin{figure}[!ht]
\begin{center}
\includegraphics[width=5cm]{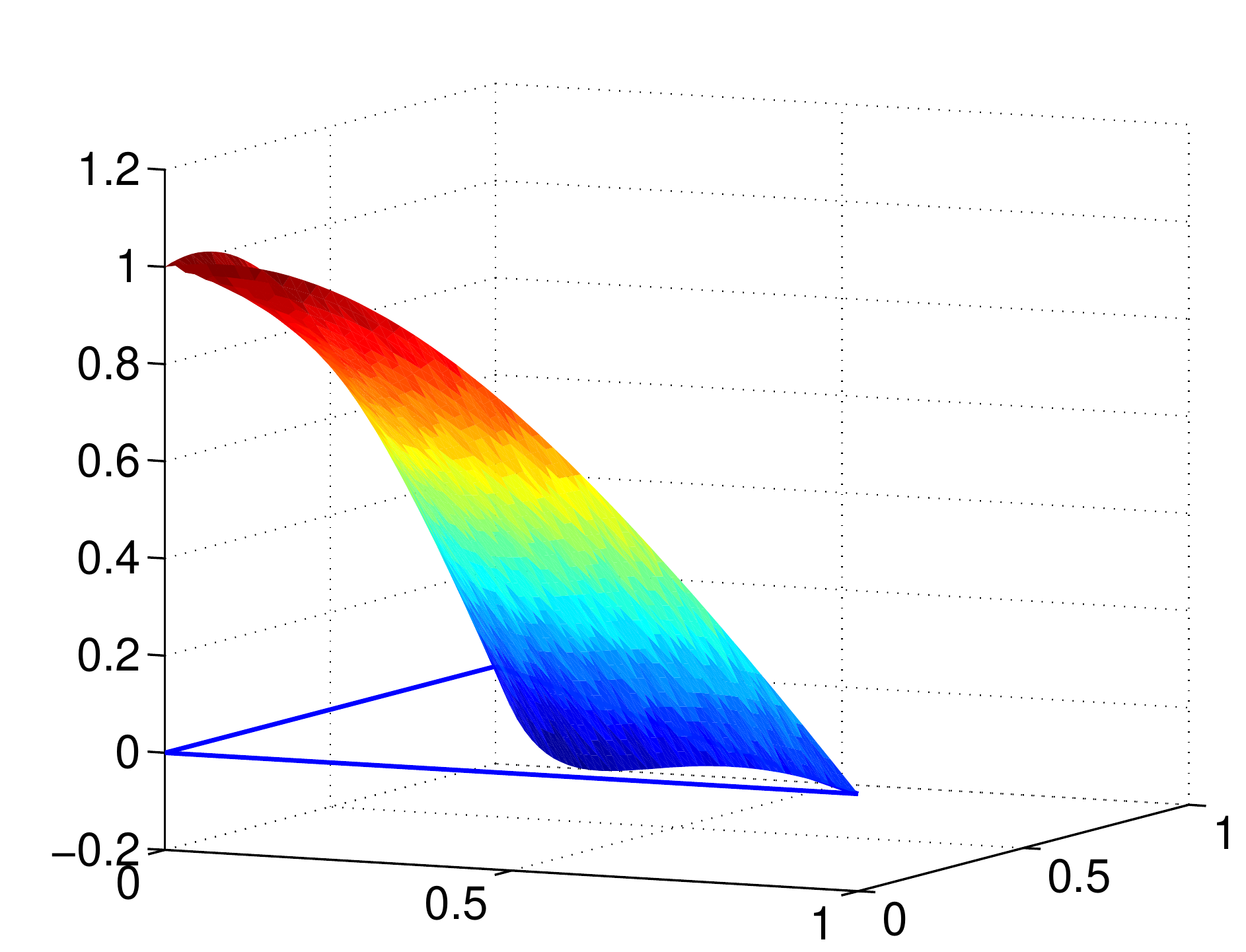}
\includegraphics[width=5cm]{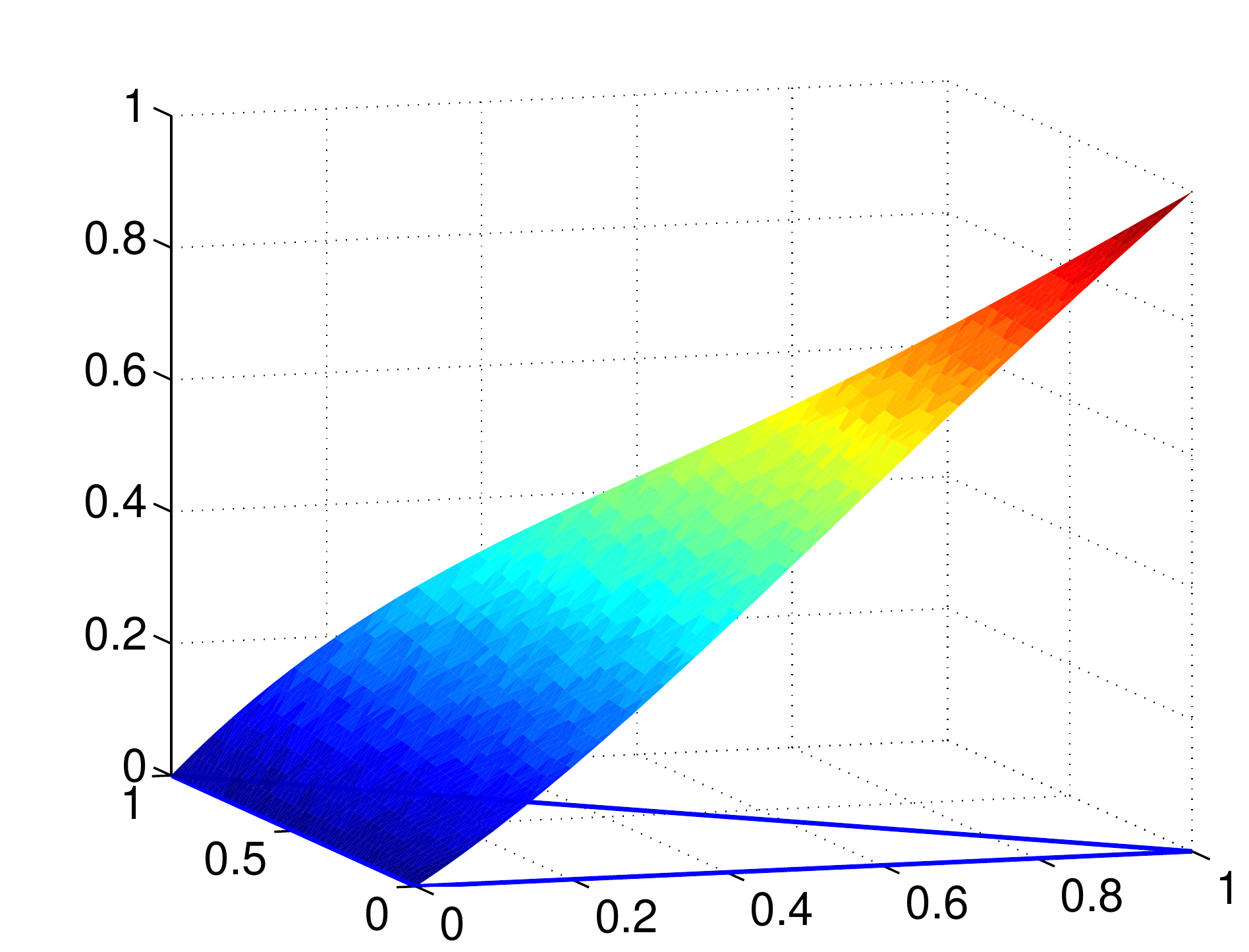}
\includegraphics[width=5cm]{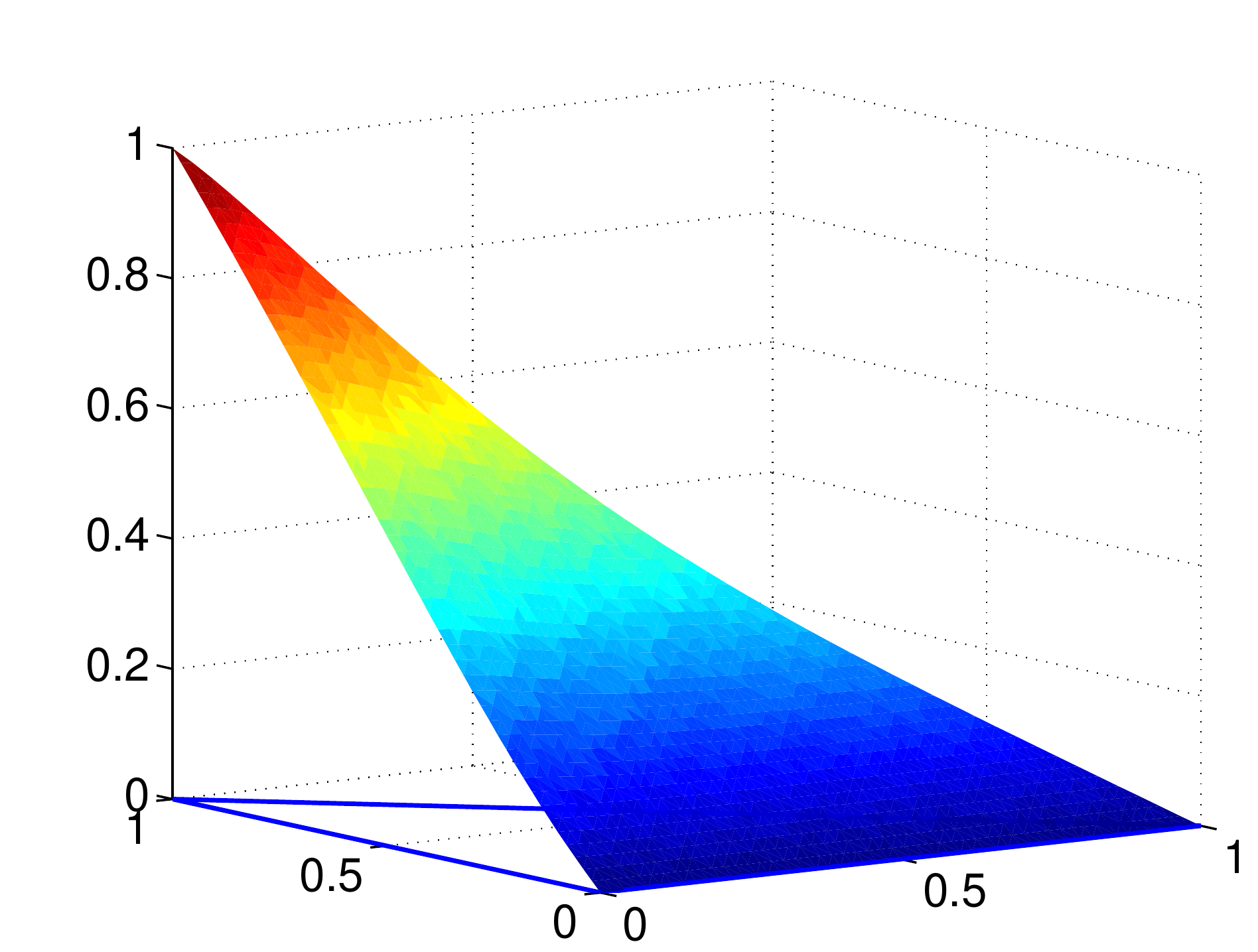}
\caption{Basis functions $u_j$ corresponding to $RT_0^0$ with interpolative constraints enforced at the three vertices. $\phi = e^{-2\sqrt{x^2 + y^2}}$.}
\label{fig:basisAN}
\end{center}
\end{figure}

\begin{figure}[!ht]
\begin{center}
\includegraphics[width=5cm]{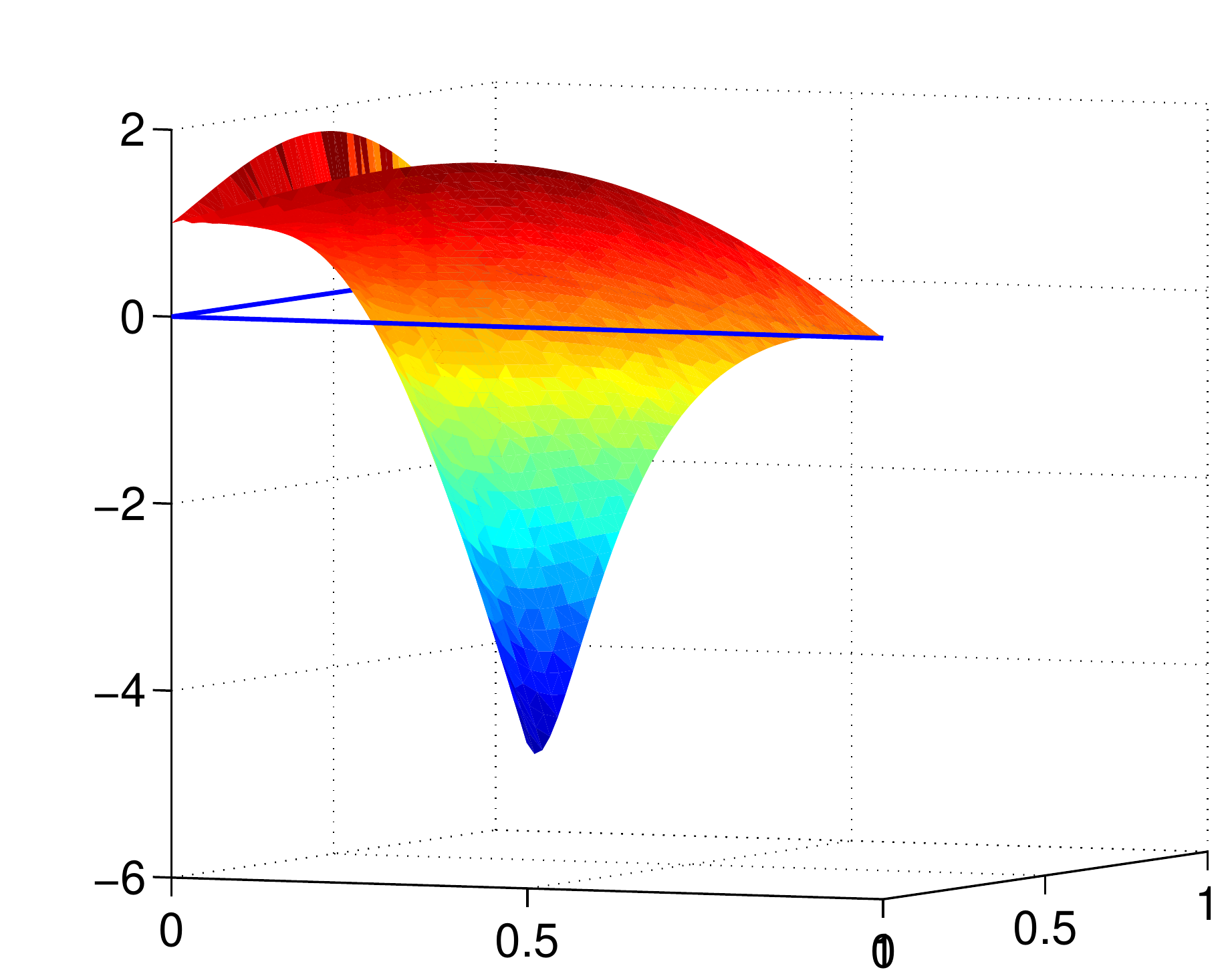}
\includegraphics[width=5cm]{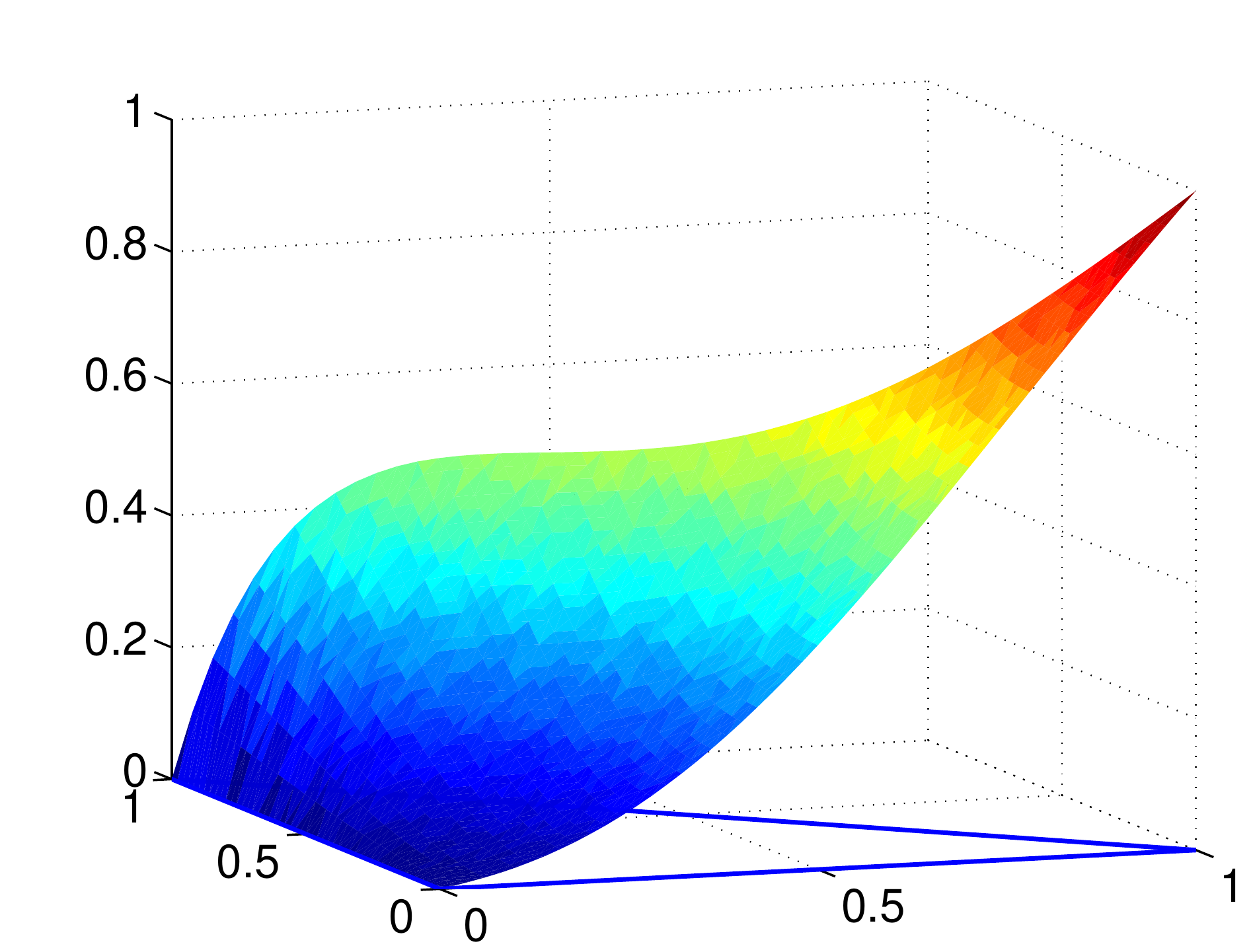}
\includegraphics[width=5cm]{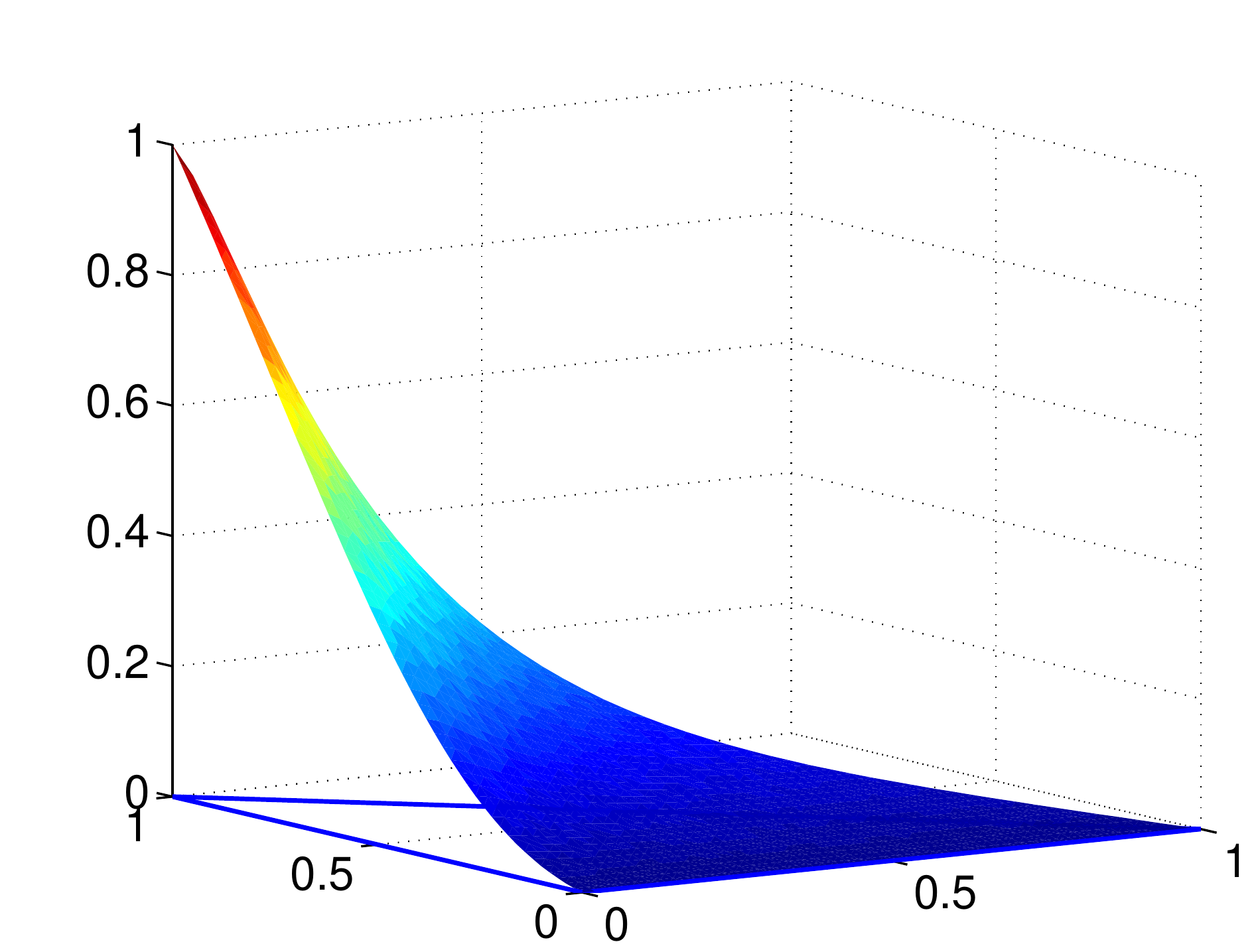}
\caption{Basis functions $u_j$ corresponding to $RT_0^0$ with interpolative constraints enforced at the three vertices. $\phi = 4e^{-2\sqrt{x^2 + y^2}}$.}
\label{fig:basisANS}
\end{center}
\end{figure}

\begin{figure}[!ht]
\begin{center}
\includegraphics[width=5cm]{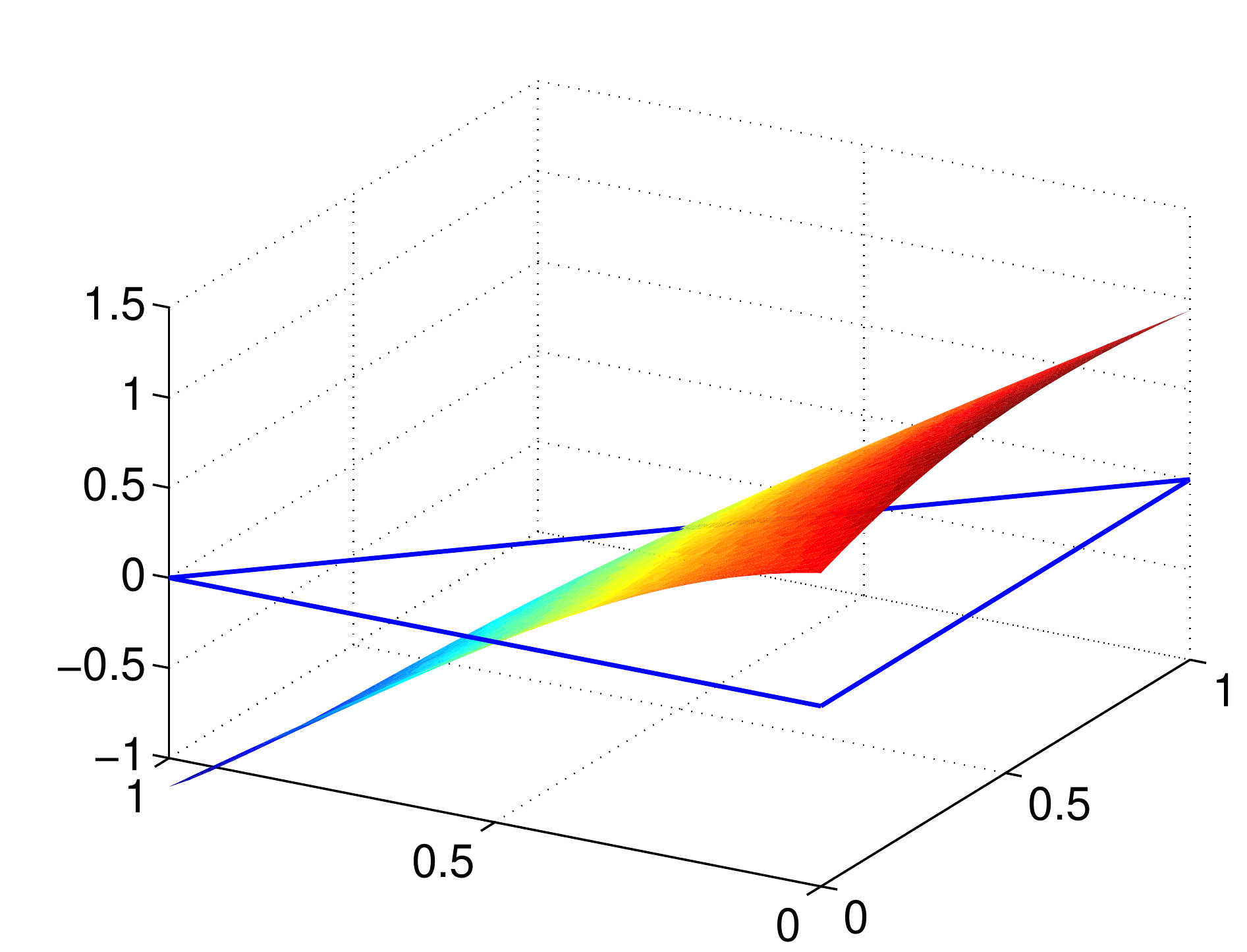}
\includegraphics[width=5cm]{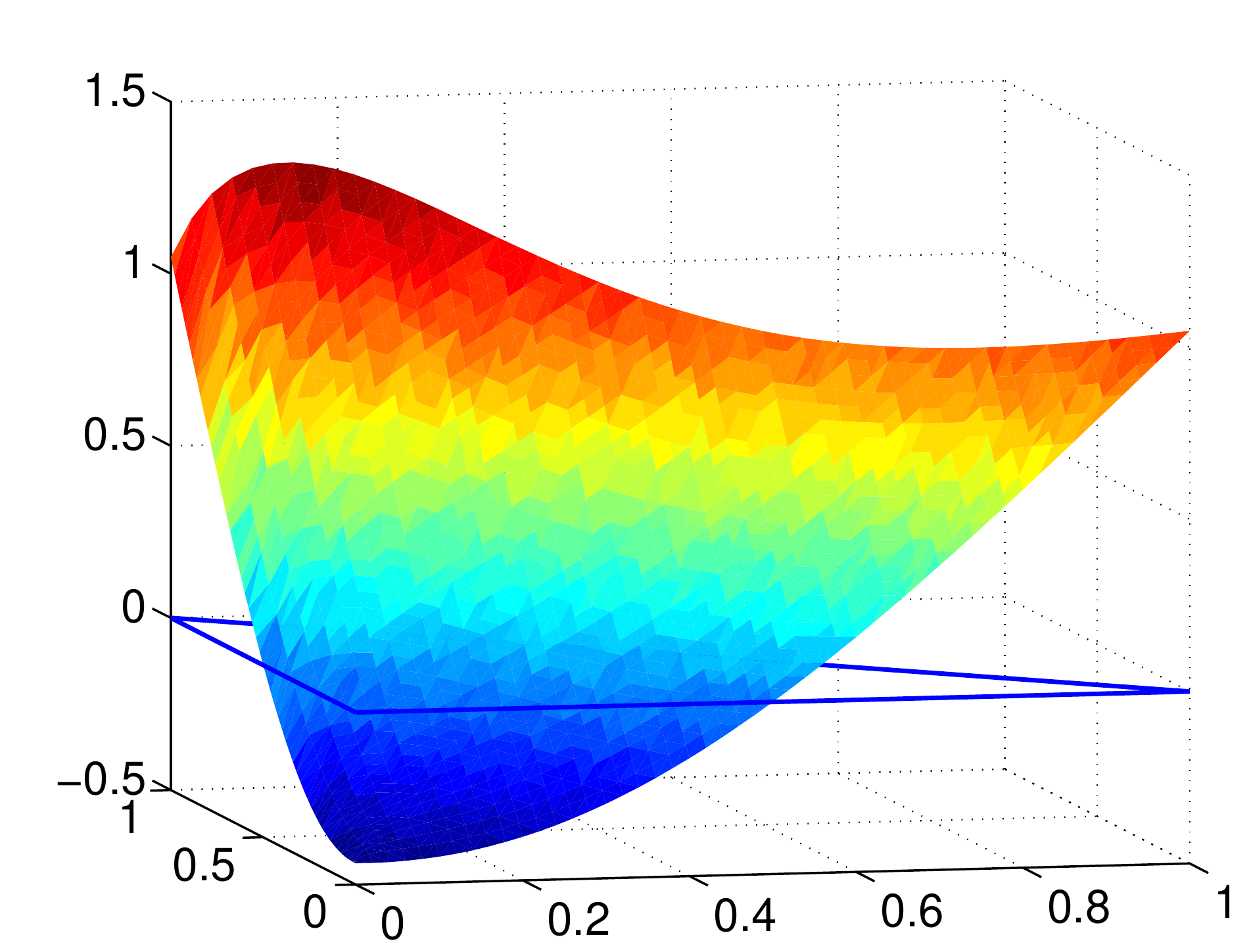}
\includegraphics[width=5cm]{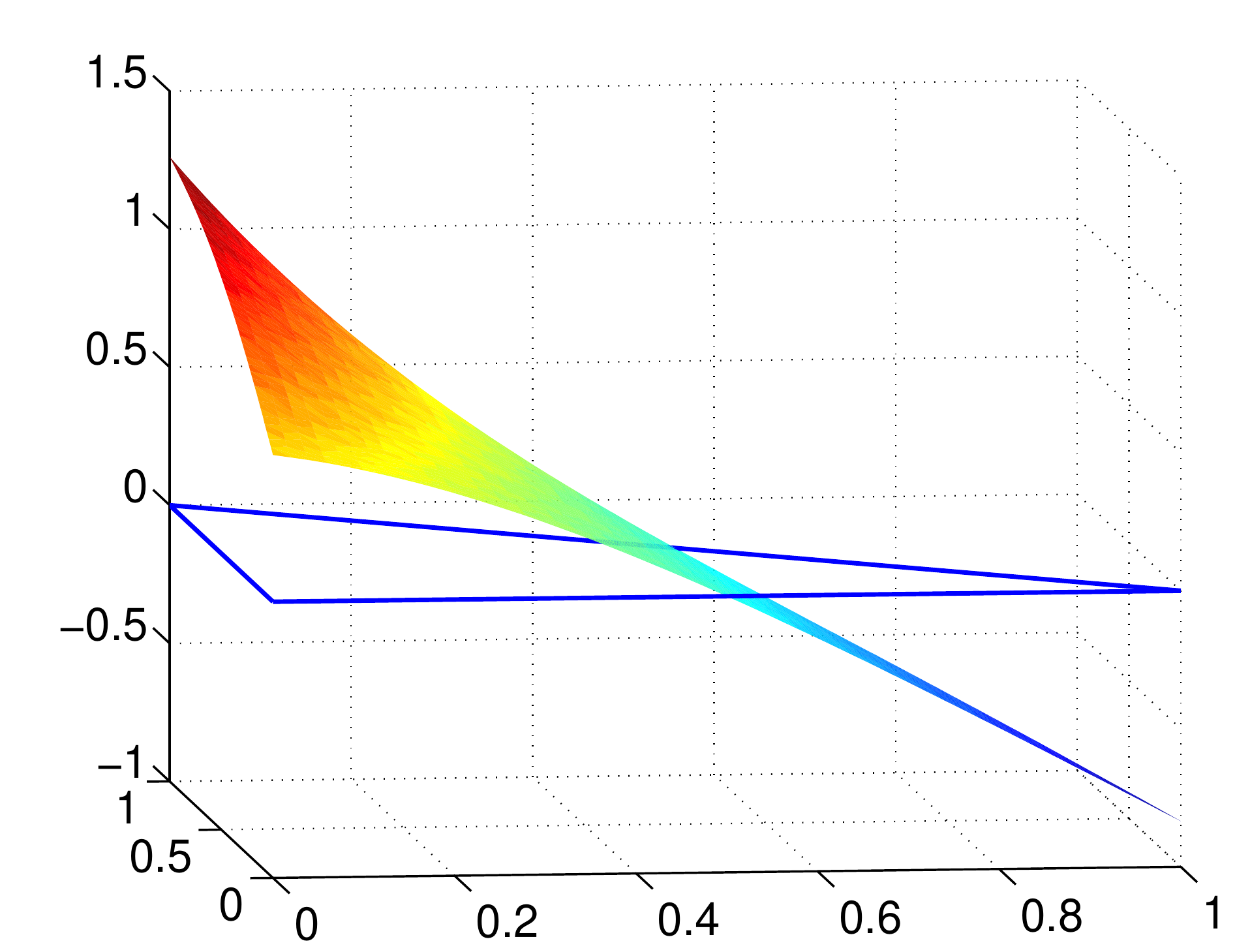}
\caption{Basis functions $u_j$ corresponding to $RT_0^0$ with interpolative constraints enforced at the three edge centers. $\phi = e^{-2\sqrt{x^2 + y^2}}$ }
\label{fig:basisBN}
\end{center}
\end{figure}

\begin{figure}[!ht]
\begin{center}
\includegraphics[width=5cm]{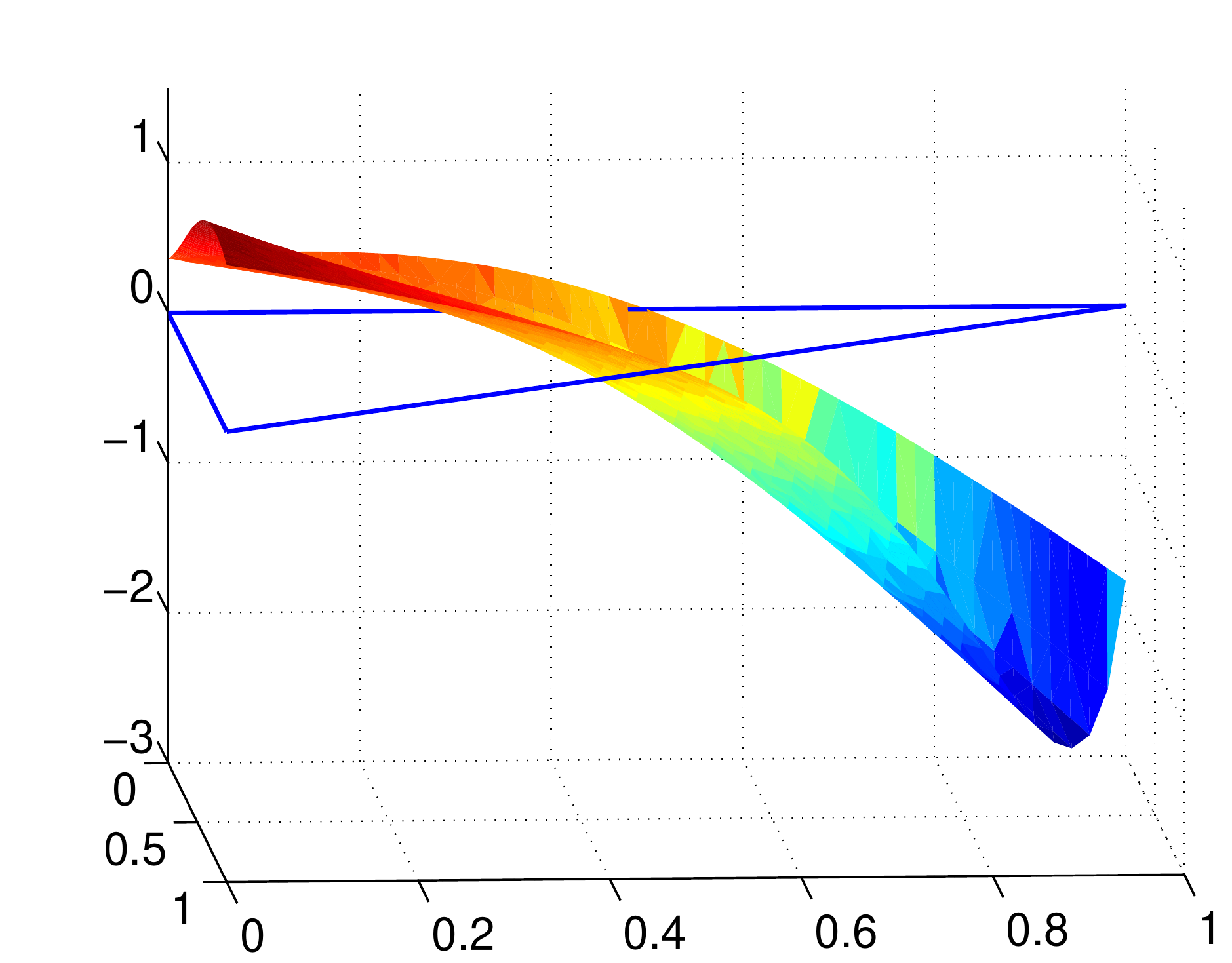}
\includegraphics[width=5cm]{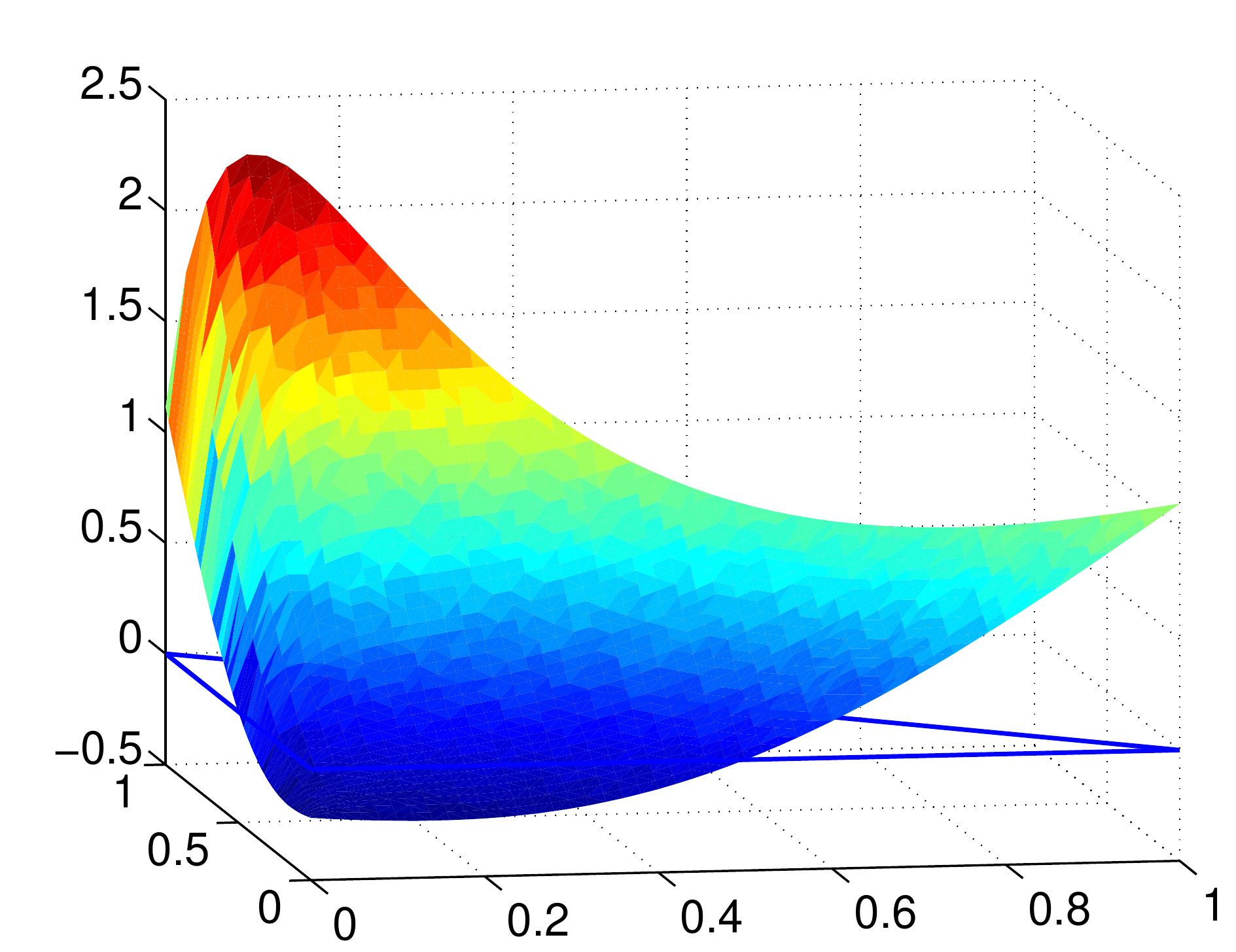}
\includegraphics[width=5cm]{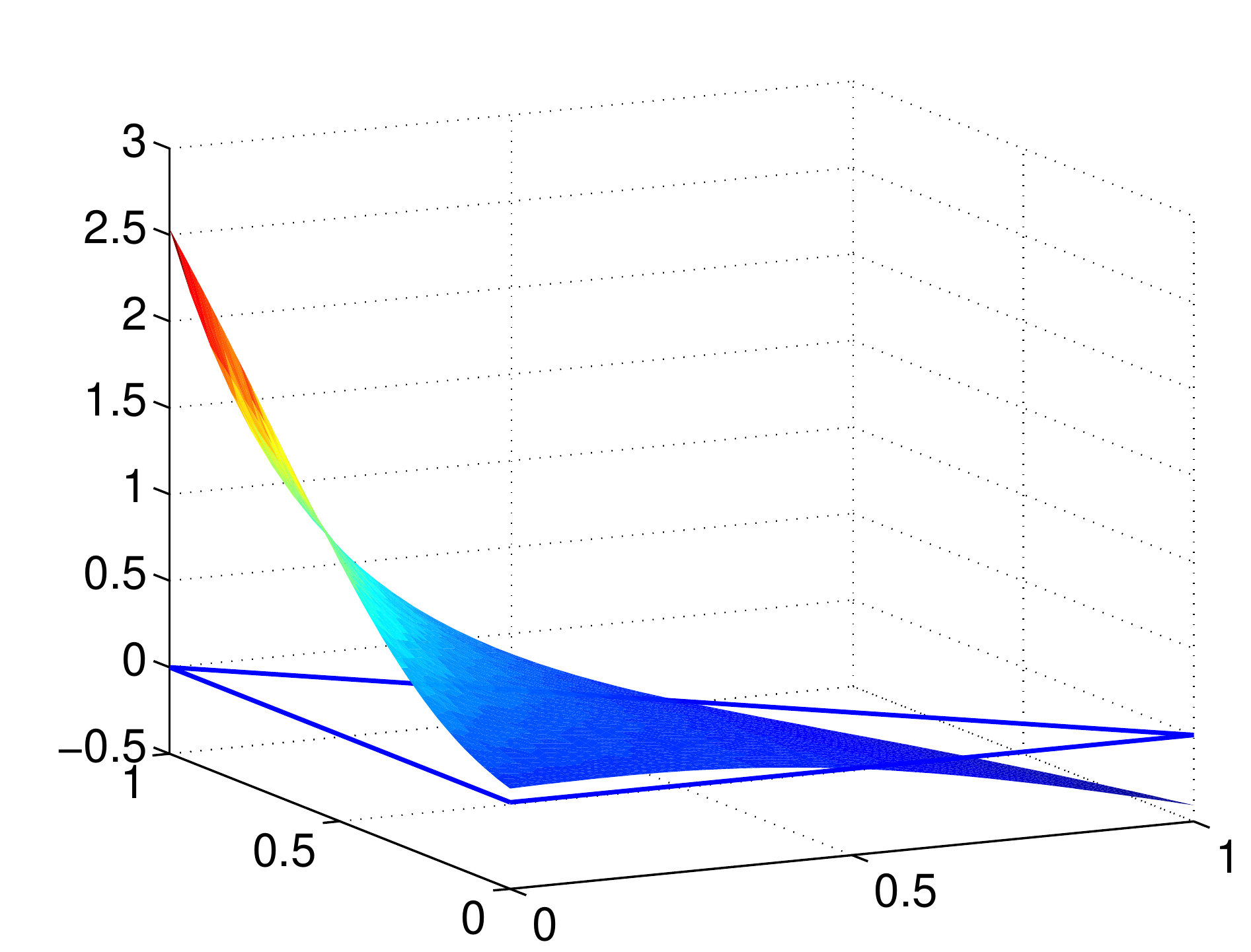}
\caption{Basis functions $u_j$ corresponding to $RT_0^0$ with interpolative constraints enforced at the three edge centers. $\phi = 4e^{-2\sqrt{x^2 + y^2}}$ }
\label{fig:basisBNS}
\end{center}
\end{figure}

\begin{figure}[!ht]
\begin{center}
\includegraphics[width=5cm]{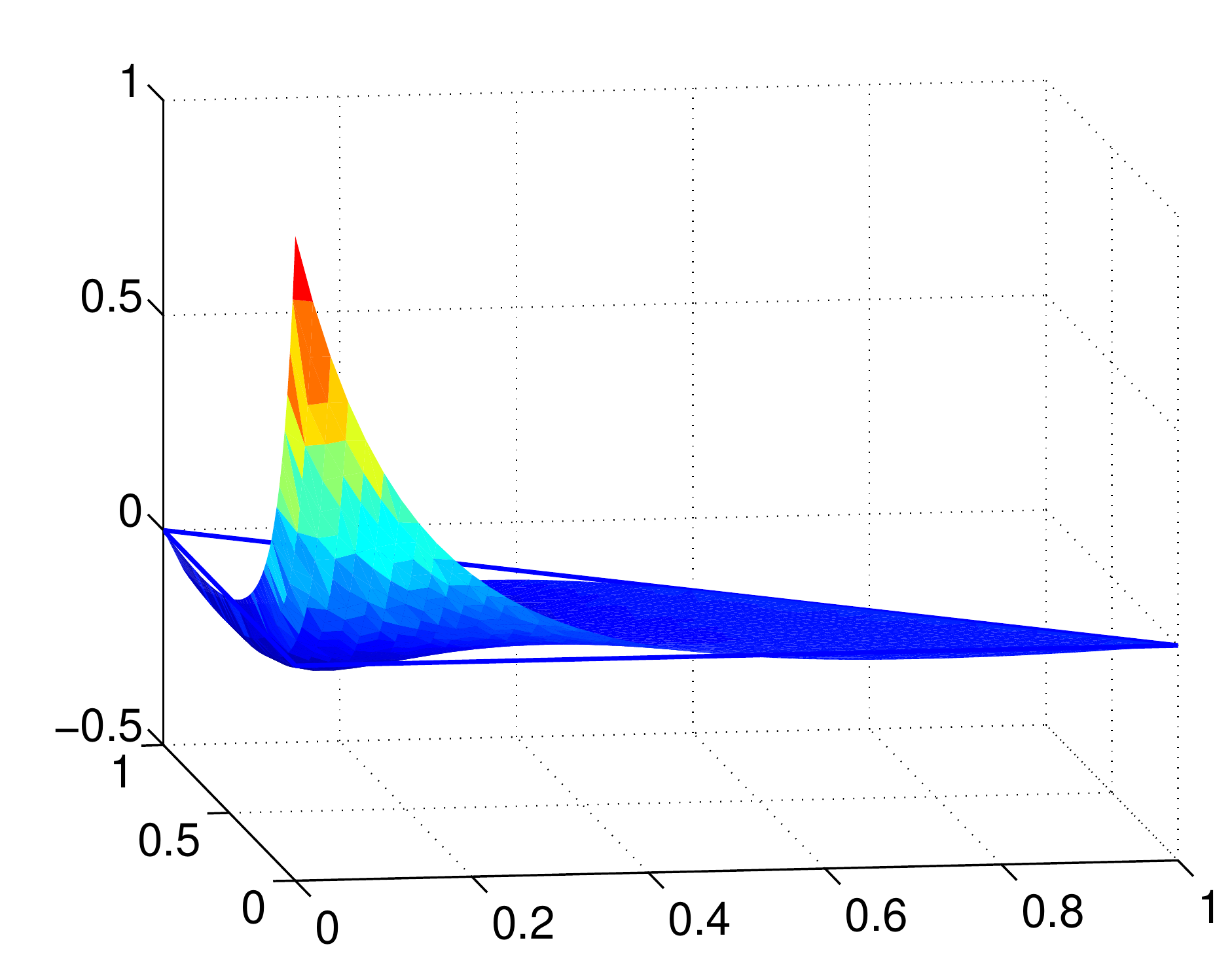}
\includegraphics[width=5cm]{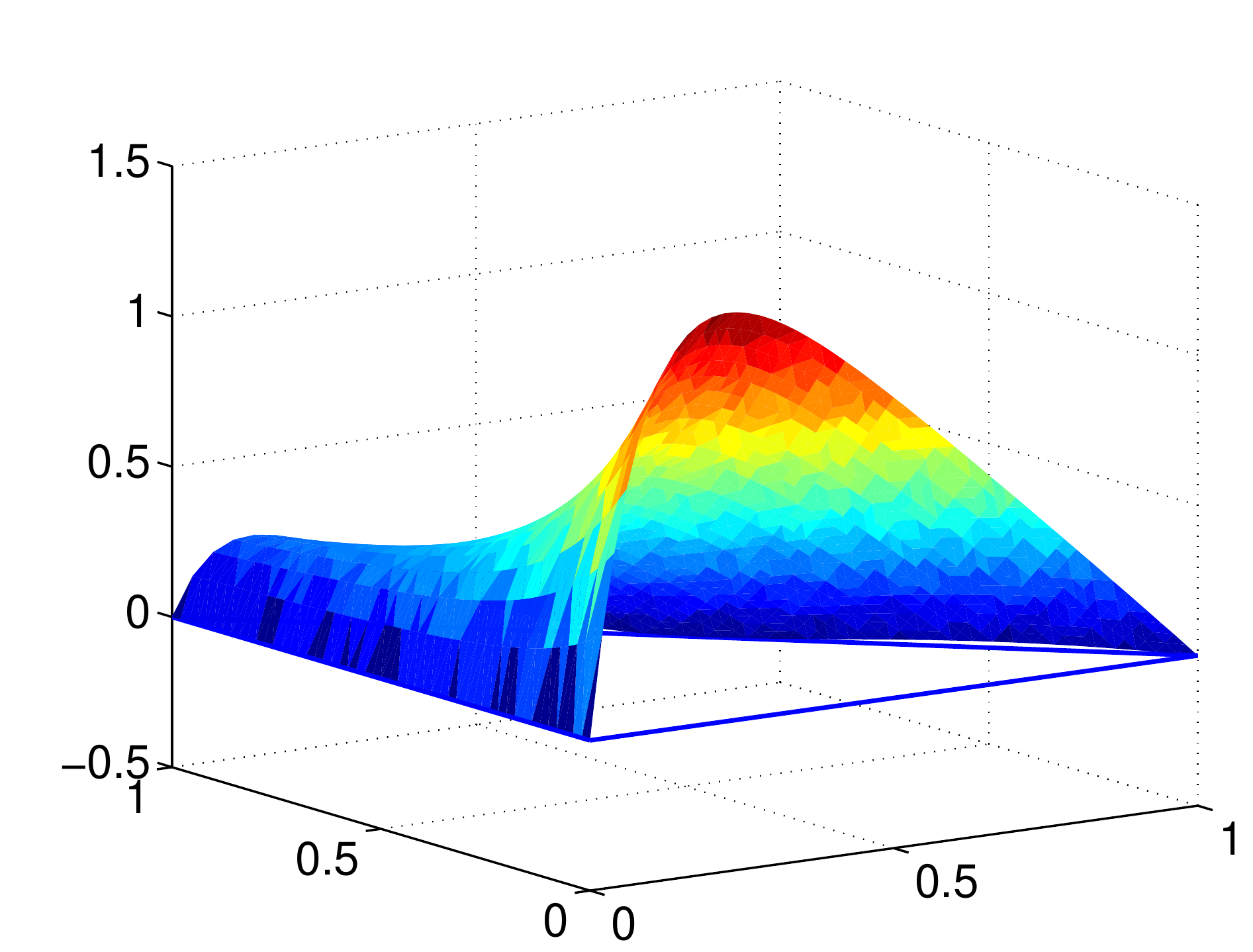}
\includegraphics[width=5cm]{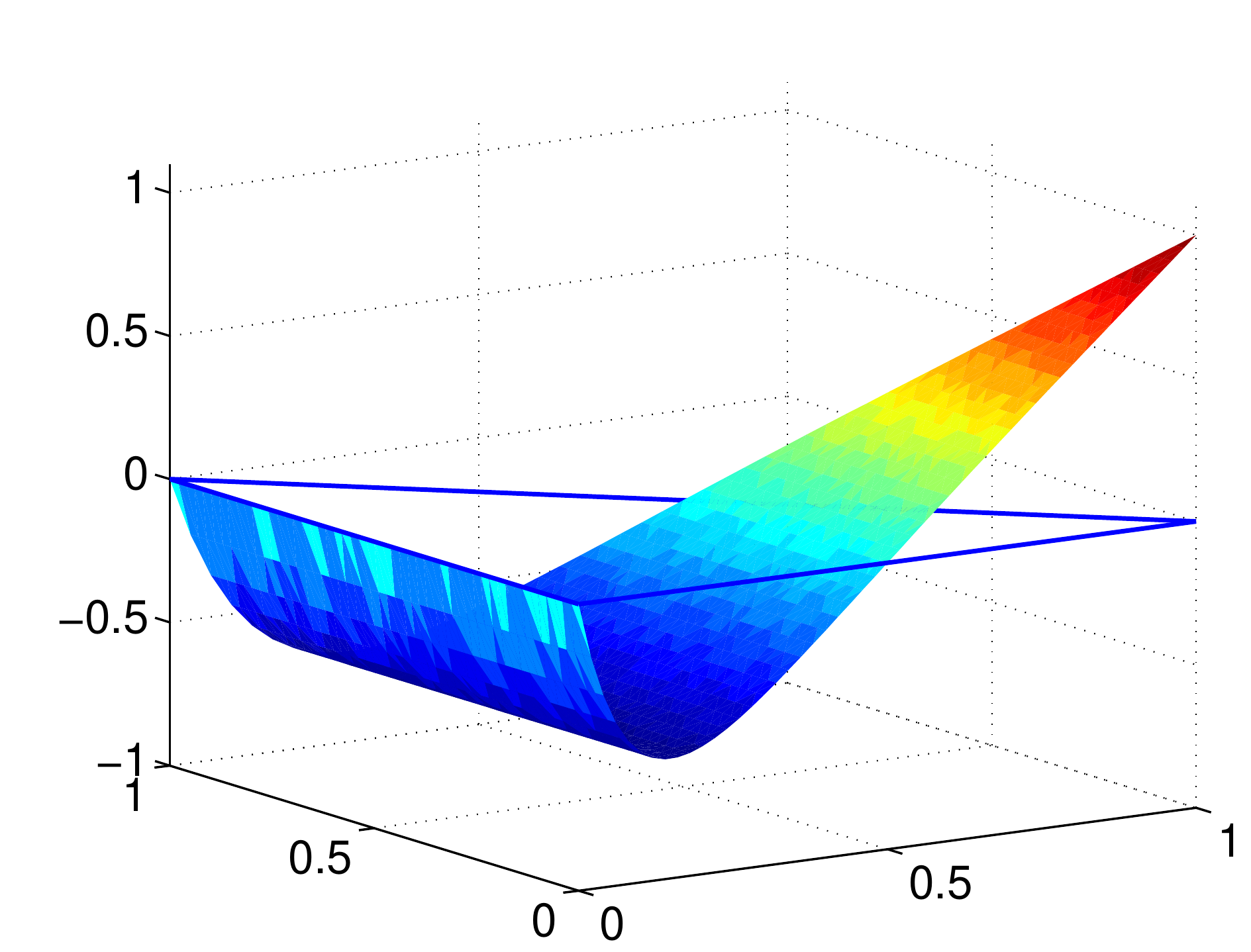}
\includegraphics[width=5cm]{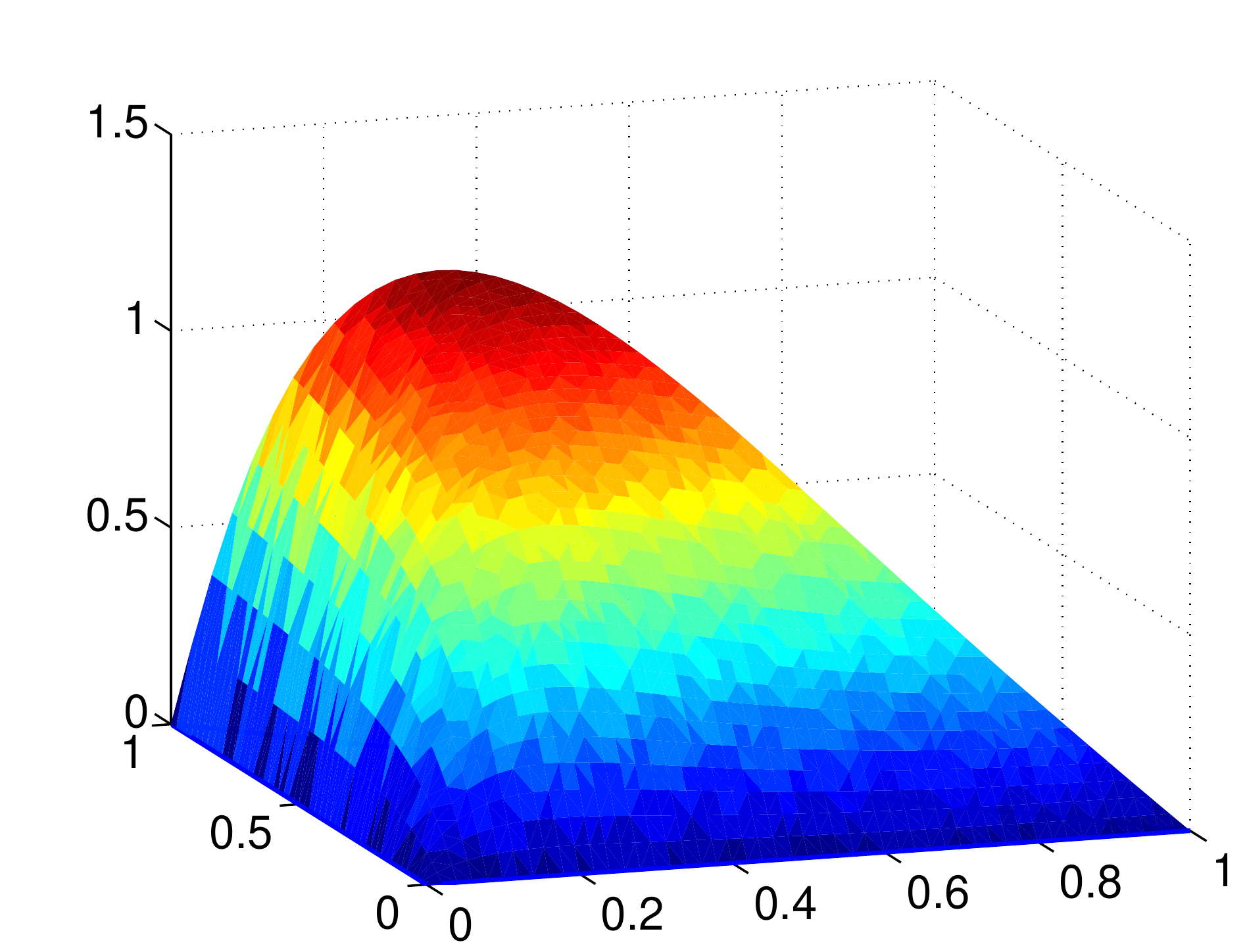}
\includegraphics[width=5cm]{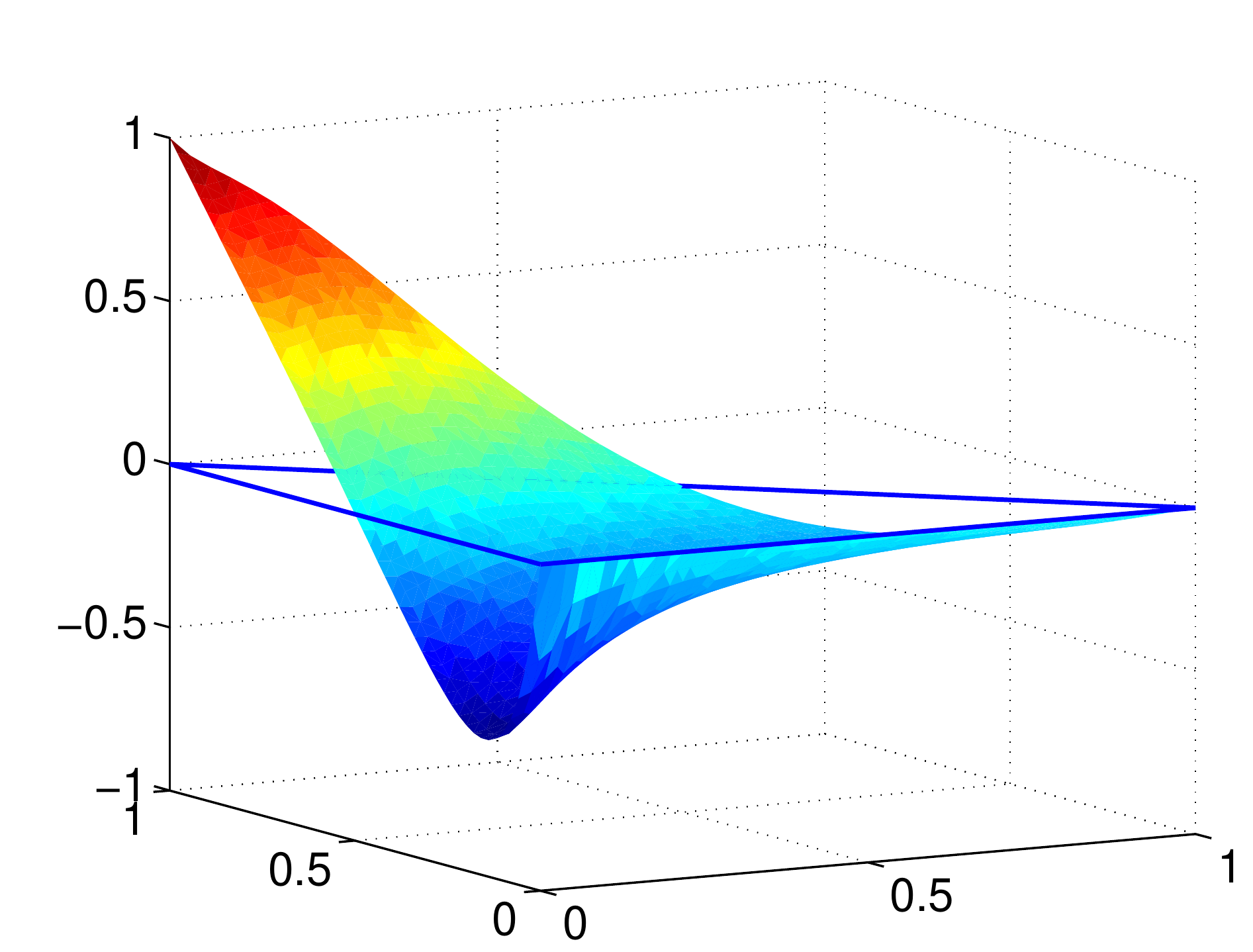}
\includegraphics[width=5cm]{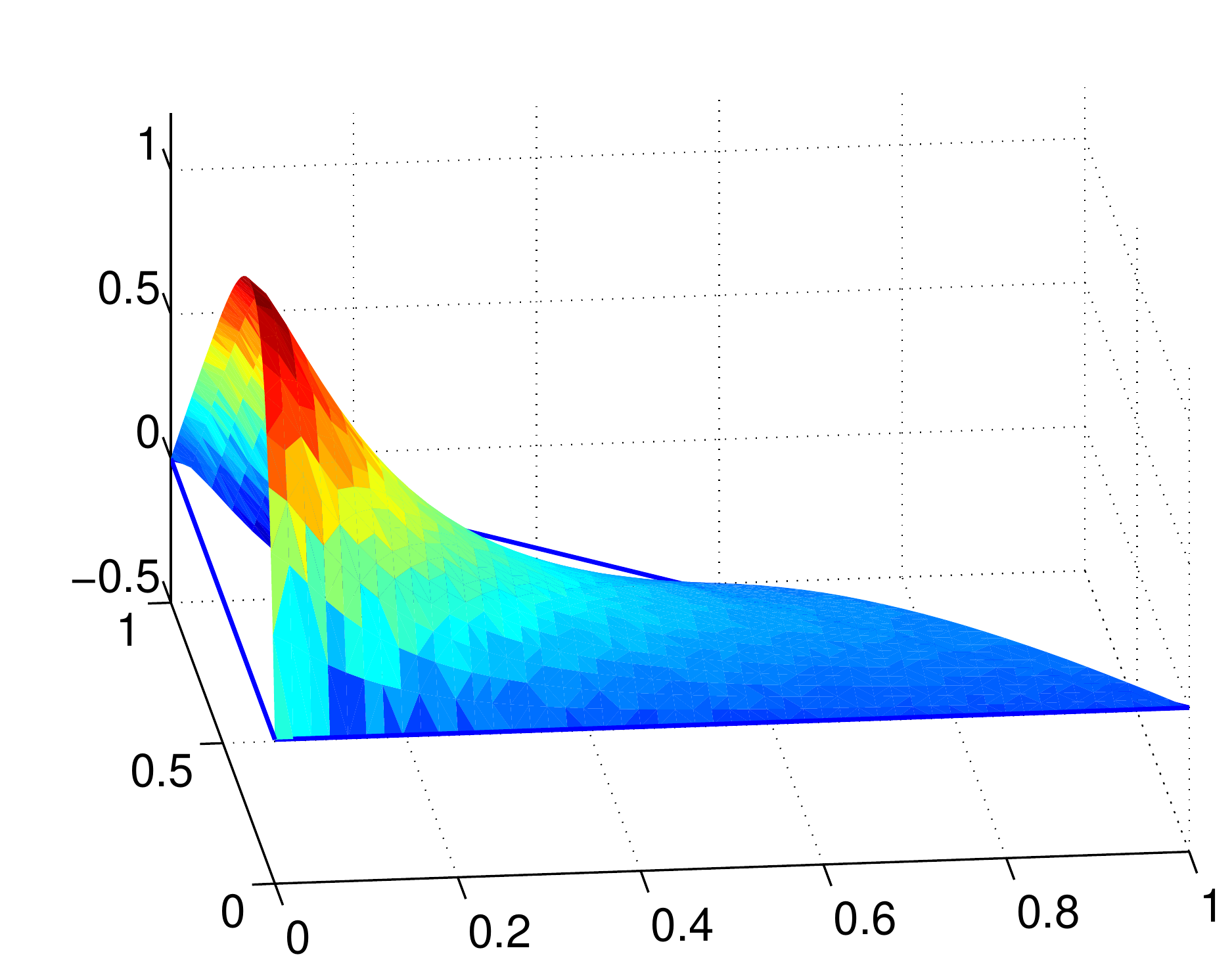}
\caption{Basis functions $u_j$ corresponding to $RT_1^0$ with interpolative constraints enforced at three vertices and three edge centers. $\phi = 4e^{-2\sqrt{x^2 + y^2}}$ }
\label{fig:basisO2}
\end{center}
\end{figure}

Our construction of the 2-D exponentially fitted methods can be illustrated by the left diagram in Figure \ref{fig:expfitidea}. Such construction
does not apply to 3-D problems because the mapping $P_{k+1} \rightarrow \hat{Q}_h$ is not one-to-one. Consequently, even a computable
basis of $\hat{V}_h$ has been found and exponentially fitted, it will be difficult in general to compute the basis for $u_h$ from these fitted
basis functions, in contrast to 2-D cases.
\begin{figure}[!ht]
\begin{center}
\includegraphics[height=2.5cm]{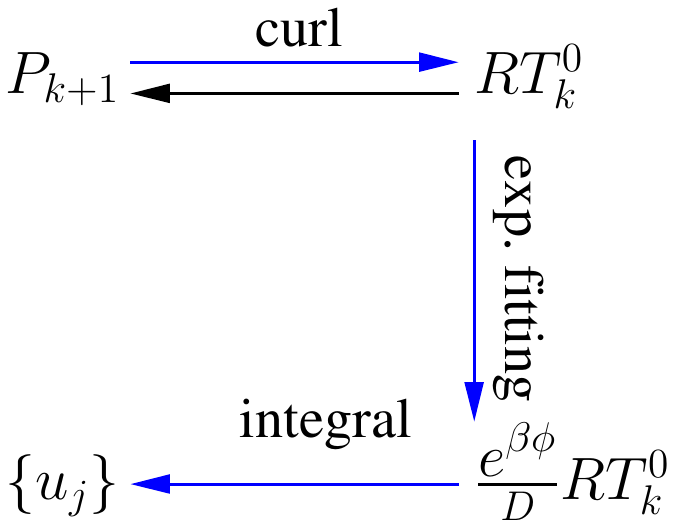} \hspace{2cm}
\includegraphics[height=2.5cm]{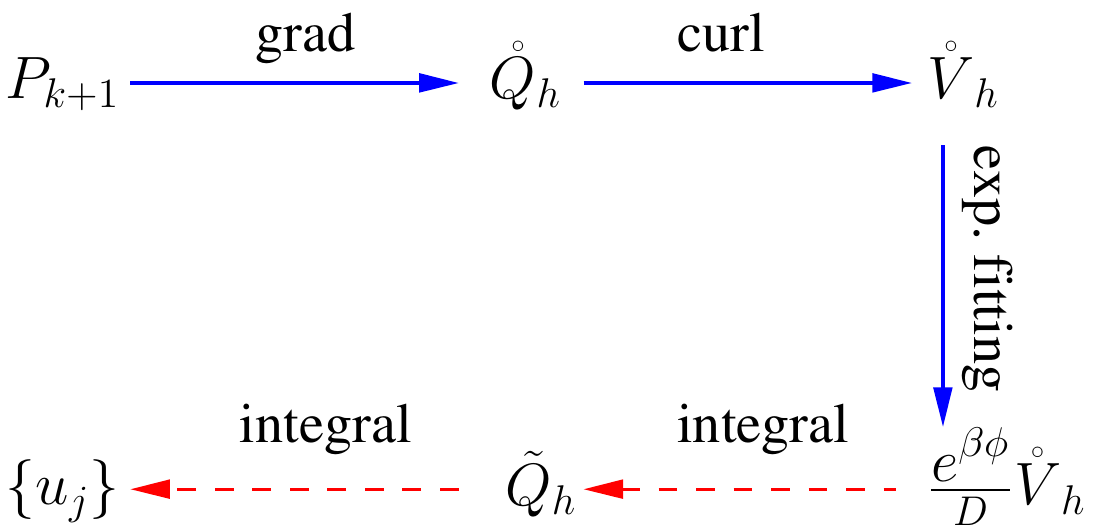}
\caption{Illustration of the construction of 2-D exponentially fitted methods (left). In 3-D (right), the mapping between $P_{k+1}$ and
$\hat{Q}_h$ is not one-to-one so constructing basis for $u_h$ from exponentially fitted $\mathring{V}_h$ for arbitrary $k$ is difficult.
}
\label{fig:expfitidea}
\end{center}
\end{figure}
It is possible however to follow the procedure indicated by (\ref{eqn:rhoexpansion_2}) to construct $\mathcal{O}(h)$ order
3-D exponentially fitted method, based on the fact that there is a simple basis of $D_h|_K=RT_0^0|_K$ in 3-D, that
is $\vecv_i = \vece_i$ on a reference tetrahedron. Hence we can use the expansion
$$ J_h(u_h)|_K = \sum_{i=1}^3 c_i \vecv_i $$
and following the procedure identical to that in subsection \ref{subsect:attempts} to finally get the following 
basis functions for $u_h|_{\hat{K}}$:
\begin{align}
u_j(x,y,z) = u_j(x_0,y_0,z_0) + &
\frac{e^{-\beta \phi}}{D} \sum_{i=1}^{3} m_{j,i}  \left (  \int_{x_0}^x e^{\beta \phi(s,y_0,z_0)} \vecv_i^x(s,y_0,z_0) ds
+ \int_{y_0}^y e^{\beta \phi(x,t,z_0)} \vecv_i^y(x,t,z_0) dt + \right.  \nonumber \\
 &  \qquad \left.  \int_{z_0}^z e^{\beta \phi(x,y,r)} \vecv_i^z(x,y,r) dr \right), \label{eqn:ucomputing_3D}
\end{align}
where $\vecv_i = (\vecv_i^x, \vecv_i^y, \vecv_i^z)^T$. The total four parameters, i.e., $u_j(x_0,y_0,z_0)$ and the three
coefficients $m_{j,i}$, will be determined by enforcing the interpolative constraints at four nodes, which can be the
four nodes of the tetrahedral or the centers of its four faces. A linear system similar to Eq.(\ref{eqn:linearsystem_m})
can be formed and its solvability can be established through a direct computation. The space $\mathrm{span}\{ u_j \}$ is
nonconforming similar to the 2-D cases.


\section{Convergence of Exponentially Fitted Methods based on $RT_0^0$} \label{sec:converg}

In this section we shall give a convergence analysis of the Galerkin finite element approximation of the problem (\ref{eqn:problem})
using the nonconforming exponentially fitted basis functions (\ref{eqn:ucomputing}) or (\ref{eqn:ucomputing_3}). We consider the
symmetrized form of the drift-diffusion equation, and adopt the approach in \cite{SaccoR1999a} to establish the $\mathcal{O}(h)$
convergence of the exponentially fitted method constructed on $D_h=RT_0^0$. Convergence of the high order methods and the 3-D method
involve more complicated error estimate of the interpolation in nonconforming finite element spaces and will be reported separately. 
Notice that the interpolative basis functions of the nonconforming finite element spaces for the Slotboom variable $\rho_h$ can be obtained 
from $u_j(x,y)$ through scaling. Let $V = H^1_0(\Omega)$ the weak solution
of the problem (\ref{eqn:problem}) can be uniquely solved from the following weak formulation
\begin{equation} \label{eqn:variational_1}
a(\rho,v) = (f,v) \quad \forall ~v \in V,
\end{equation}
where the coercive, symmetric bilinear form $a(\rho,v): V \times V \rightarrow \mathbb{R}$ is given
$$ a(\rho,v) = \int_{\Omega} J(\rho) \cdot \nabla v dx = \int_{\Omega} D e^{-\beta \phi} \nabla \rho \cdot \nabla v dx.$$
For the purpose of Galerkin finite element approximation of this weak formulation we define the finite element
space
$$ V_h = \{ w \in L^2(\Omega): w \in \mathrm{span} \{ \rho_j \} ~\forall ~ K \in \mathcal{T}; w = 0 ~ \mbox{on}~ \partial \Omega \}.$$
By construction $\rho_j$ is continuous only at selected nodes but in general discontinuous on edges, hence $V_h$ is not a subspace of
$V$ in general. Following the standard treatments for nonconforming finite element methods (E.g., Chapter 10 on \cite{BrennerScottFEMv3})
we introduce the space $V + V_h$ in which we can define discrete bilinear and linear form, as
\begin{align}
a_h(\rho_h, v_h) & = \sum_{K \in \mathcal{T}_h}
\int_{K} D e^{-\beta \phi} \nabla \rho_h \cdot \nabla v_h, \quad \forall ~\rho_h, v_h \in V + V_h, \label{eqn:new_bilinear} \\
(f,v_h)_h & = \sum_{K \in \mathcal{T}_h} \int_{K} f v_h, \quad \forall  v_h \in V_h. \label{eqn:new_linear}
\end{align}
We equip the space $V + V_h$ with a norm
\begin{equation} \label{eqn:energynorm}
\| v \|_h = \left (  \sum_{K \in \mathcal{T}_h}  | v |_{1,K}^2  \right)^{1/2}.
\end{equation}
It can be seen that this norm becomes the semi-norm $|v|_{1,\Omega}$ for $v \in V$, which is itself a norm due to the Poinc\'{a}re
inequality. To check that $\| v\|_h =0$ if and only if $v=0$ for $v \in V_h$ we notice that $v$ must be a constant for
$|v|_{1, K}= 0$. Since $v \in V_h$ is continuous at nodes by construction the constant must be the same for all elements,
and hence shall be zero for $v = 0$ on $\partial \Omega$. It follows immediately that $a_h(\rho_h,v_h)$ is continuous and
coercive on $V+V_h$ because
$$ |a_h(v_h, w_h) | = \sum_{K \in \mathcal{T}_h} \int_{K} D e^{-\beta \phi} \nabla v_h \cdot \nabla w_h \le
M   \| v_h \|_h \| w_h \|_h, \quad \forall ~ v_h, w_h \in V+V_h, $$
and
$$ a_h(v_h, v_h ) =  \sum_{K \in \mathcal{T}_h} \int_{K} D e^{-\beta \phi} \nabla \rho_h \cdot \nabla v_h \ge  m \|v_h\|_h^2, \quad \forall  ~ v_h \in V+V_h,$$
where
$$ M = e^{-\beta \min_{\vecx \in \Omega} | \phi(\vecx)|}, \quad m = e^{-\beta \max_{\vecx \in \Omega}{|\phi(\vecx)|}}. $$
The nonconforming Galerkin finite element approximation of problem (\ref{eqn:driftdiffusioneq_sym}) finally
reads: \textit{Find $\rho_h \in V_h$ such that}
\begin{equation} \label{eqn:weak_final}
a_h(\rho_h, v_h ) = (f,v_h)_h, \quad \forall ~ v_h \in V_h.
\end{equation}
There are two major components in the convergence analysis of the nonconforming exponentially fitted finite element method, one is the
interpolation error in $V_h$, and the other is the consistency error arising from the nonconformity of the method, according to the
second Strang lemma. The analysis in \cite{SaccoR1999a} follows this standard approach. Estimation of the interpolation error is obtained
through a decomposition
\begin{equation} \label{eqn:interpolation_decomposition}
\norm{\rho - \Pi_h \rho} \le \norm{\rho - \Pi_1 \rho}_h + \norm{\Pi_h \rho - \Pi_1 \rho}_h,
\end{equation}
where $\Pi_h \rho$ is the interpolation of $\rho$ in $V_h$, and $\Pi_1 \rho$ is the $P_1$ \textit{conforming interpolation} of $\rho$ in
$\mathcal{T}_h$. This decomposition is applicable to our constructions with interpolative constraints enforced at the same set of
nodes as those for conforming Lagrange finite element spaces. If the edge nodes are chosen for the constraint enforcement we will
have to adopt a decomposition alternative to (\ref{eqn:interpolation_decomposition})
\begin{equation} \label{eqn:interpolation_decomposition_nonconforming}
\norm{\rho - \Pi_h \rho} \le \norm{\rho - \tilde{\Pi}_1 \rho}_h + \norm{\Pi_h \rho - \tilde{\Pi}_1 \rho}_h,
\end{equation}
here $\tilde{\Pi}_1 \rho$ is the $P_1$ \textit{nonconforming} interpolation of $\rho$ in $\mathcal{T}_h$. The first terms in
these two cases have a similar estimate (\cite{HandbookNA_FEM_I}, P. 130 and \cite{CrouzeixRaviart})
\begin{equation} \label{eqn:interpolation_error}
\norm{\rho - \Pi_1 \rho}_h \le C_1 h |\rho|_{2, \Omega}, \quad \norm{\rho - \tilde{\Pi}_1 \rho}_h \le C_2 h |\rho|_{2, \Omega}
\end{equation}
for different generic constants $C_1,C_2$. Consequently for both cases the following result concerning the convergence of
the solution $\rho_h$ of Eq.(\ref{eqn:weak_final}) can be established.
\begin{theorem}
Let $\rho$ be the unique weak solution of Problem (\ref{eqn:driftdiffusioneq_sym}) in $H^2(\Omega) \cap H^1_0(\Omega)$ and let
$\{ \mathcal{T}_h \}$ be a family of regular triangulations of $\bar{\Omega}$. Then there exist positive constants $C_1(\phi),
C_2(\phi)$ and $C_3(\phi)$ that depend only on $\phi$, such that
\begin{equation} \label{eqn:main_estimate}
\| \rho - \rho_h \|_h \le C_1(\phi) h | \rho|_{2, \Omega} + C_2(\phi) \| \nabla \phi \|_{\infty, \Omega}
| \Pi_1 \rho |_{1, \Omega} + C_3(\phi) h \| \nabla \phi \|_{\infty, \mathcal{T}_h} \| \rho \|_{2, \Omega}.
\end{equation}
\end{theorem}
\begin{proof}
In \cite{SaccoR1999a}.
\end{proof}

\section{Summary and Future Work} \label{sec:summary}
We proposed a general approach to construct exponentially fitted methods for solving 2-D drift-diffusion equations.
Our constructions are based on the divergence-free subspace of $\hdiv$, usually $RT_k^0$. The basis functions $u_j$ for the
density function are computed from the basis of $RT_k^0$, and through enforcement of interpolative constraints at selected
node sets, giving rise to nonconforming finite element spaces. We showed that our approach can reproduce some of the
previous constructions of the exponentially fitted methods, and will recover the standard conforming or nonconforming
finite element spaces when the electric potential is zero. Thanks to the one-one mapping between $RT_k^0$ and $P_{k+1}$,
our approach can be used to develop arbitrary high-order 2-D exponentially fitted methods from $RT_k^0$. Extension of our
approach for high-order 3-D methods is difficult because there does not exist a similar one-one correspondence between
$k^{th}$ order divergence-free vector space and $P_{k+1}$, but construction of first-order 3-D exponentially fitted method
is possible.

General convergence theory for the high-order methods can be established following decomposition (\ref{eqn:interpolation_decomposition}),
the estimate of $P_k$ conforming interpolation errors, the estimate of $\Pi_h \rho - \Pi_k \rho$, and the consistency error of high-order nonconforming
finite element methods. These results will be reported in a forthcoming paper. We will also conduct extensive numerical experiments of the methods,
in particular the high order methods, and apply them to realistic biomolecular drift-diffusion problems.


\begin{thebibliography}{10}

\bibitem{AllenD1955}
D.~N. Allen and R.~V. Southwell.
\newblock Relaxation method applied to determine the motion, in two dimensions,
  of viscous fluid past a fixed cylinder.
\newblock {\em Q.\ J.\ Mech.\ Appl.\ Math.}, 8(2):129--145, 1955.

\bibitem{AmatoreC2009b}
C.~Amatore, O.~V. Klymenko, A.~I. Oleinick, and I.~Svir.
\newblock Diffusion with moving boundary on spherical surfaces.
\newblock {\em Chemphyschem.}, 10:1593--1602, 2009.

\bibitem{AmatoreC2009a}
C.~Amatore, A.~I. Oleinick, O.~V. Klymenko, and I.~Svir.
\newblock Theory of long-range diffusion of proteins on a spherical biological
  membrane: Application to protein cluster formation and actin-comet tail
  growth.
\newblock {\em Chemphyschem.}, 10:1586--1592, 2009.

\bibitem{AngermannL2003b}
Lutz Angermann and Song Wang.
\newblock On convergence of the exponentially fitted finite volume method with
  an anisotropic mesh refinement for a singularly perturbed
  convection-diffusion equation.
\newblock {\em Computational Methods in Applied Mathematics}, 3:493--512, 2003.

\bibitem{AngermannL2003a}
Lutz Angermann and Song Wang.
\newblock Three-dimensional exponentially fitted conforming tetrahedral finite
  elements for the semiconductor continuity equations.
\newblock {\em Applied Numerical Mathematics}, 46(1):19 -- 43, 2003.

\bibitem{DouglasA2000a}
Douglas~N. Arnold, Richard~S. Falk, and Ragnar Winther.
\newblock Multigrid in h(div) and h(curl).
\newblock {\em Numer.\ Math.}, 85:197--218, 2000.

\bibitem{BrennerScottFEMv3}
S.~C. Brenner and L.~R. Scott.
\newblock {\em The Mathematical Theory of Finite Element Methods}.
\newblock Springer, 2008.

\bibitem{BDMElements}
F.~Brezzi, J.~Douglas, and L.~Marini.
\newblock Two families of mixed finite elements for second order elliptic
  problem.
\newblock {\em Math.\ Comp.}, 47:217--235, 1985.

\bibitem{mixedandhybridFEM}
F.~Brezzi and M.~Fortin.
\newblock {\em Mixed and Hybrid Finite Elements}.
\newblock Springer-Verlag, 1991.

\bibitem{BrezziF1989a}
F.~Brezzi, L.~D. Marini, and P.~Pietra.
\newblock Two-dimensional exponential fitting and applications to
  drift-diffusion models.
\newblock {\em SIAM J.\ Numer.\ Anal.}, 26:1342--1355, 1989.

\bibitem{HillariretC2004a}
Y.~J.~Peng C.~Chainais-Hillairet.
\newblock Finite volume approximation for degenerate drift-diffusion system in
  several space dimensions.
\newblock {\em Mathematical Models and Methods in Applied Sciences},
  14(03):461--481, 2004.

\bibitem{ChaudhryJ2011b}
Jehanzeb~Hameed Chaudhry, Jeffrey Comer, Aleksei Aksimentiev, and Luke~N.
  Olson.
\newblock A finite element method for modified {Poisson-Nernst-Planck}
  equations to determine ion flow though a nanopore.
\newblock Preprint.

\bibitem{HandbookNA_FEM_I}
P.G. Ciarlet and J.L. Lions, editors.
\newblock {\em Handbook of Numerical Analysis: Finite Element Methods (Part
  1)}.
\newblock North-Holland, 1991.

\bibitem{CrouzeixRaviart}
M.~Crouzeix and P.A. Raviart.
\newblock Conforming and nonconforming finite element methods for solving the
  stationary {Stokes} equations.
\newblock {\em RAIRO Anal. Numer.}, 7:33--76, 1973.

\bibitem{deFalcoC2009a}
Carlo de~Falco, Joseph~W. Jerome, and Riccardo Sacco.
\newblock Quantum-corrected drift-diffusion models: Solution fixed point map
  and finite element approximation.
\newblock {\em J.\ Comput.\ Phys.}, 228(5):1770 -- 1789, 2009.

\bibitem{GattiE1998}
E.~Gatti, S.~Micheletti, and R.~Sacco.
\newblock A new {Galerkin} framework for the drift-diffusion equation in
  semiconductors.
\newblock {\em East-West J. Numer. Math.}, 6:101--135, 1998.

\bibitem{HegartyA1993a}
Alan~F. Hegarty, Eugene O'Riordan, and Martin Stynes.
\newblock A comparison of uniformly convergent difference schemes for
  two-dimensional convection—diffusion problems.
\newblock {\em J.\ Comput.\ Phys.}, 105(1):24 -- 32, 1993.

\bibitem{JeromeSemiconductorbook}
J.~W. Jerome.
\newblock {\em Analysis of Charge Transport: A Mathematical Study of
  Semiconductor Devices}.
\newblock Springer, 1996.

\bibitem{KiselevV2011a}
Vladimir~Yu. Kiselev, Marcin Leda, Alexey~I. Lobanov, Davide Marenduzzo, and
  Andrew~B. Goryachev.
\newblock Lateral dynamics of charged lipids and peripheral proteins in
  spatially heterogeneous membranes: Comparison of continuous and monte carlo
  approaches.
\newblock {\em J. Chem. Phys.}, 135:155103, 2011.

\bibitem{LazarovR1996a}
R.~D. Lazarov, Ilya~D. Mishev, P.~S. Vassilevski, and L.
\newblock Finite volume methods for convection-diffusion problems.
\newblock {\em SIAM J.\ Numer.\ Anal.}, 33:31--55, 1996.

\bibitem{Lu07c}
B.~Z. Lu, Y.~C. Zhou, Gary~A. Huber, Steve~D. Bond, Michael~J. Holst, and J.~A.
  McCammon.
\newblock Electrodiffusion: A continuum modeling framework for biomolecular
  systems with realistic spatiotemporal resolution.
\newblock {\em J. Chem. Phys.}, 127:135102, 2007.

\bibitem{ZhouY2010a}
Benzhuo Lu, M.~J. Holst, J.~A. McCammon, and Y.~C. Zhou.
\newblock {Poisson-Nernst-Planck Equations} for simulating biomolecular
  diffusion-reaction processes {I}: {Finite element} solutions.
\newblock {\em J.\ Comput.\ Phys.}, 229:6979--6994, 2010.

\bibitem{ZhouY2011a}
Benzhuo Lu and Y.~C. Zhou.
\newblock {Poisson-Nernst-Planck} equations for simulating biomolecular
  diffusion-reaction processes {II}: Size effects on ionic distributions and
  diffusion-reaction rates.
\newblock {\em Biophys. J.}, 100:2475--2485, 2011.

\bibitem{MillerJ1994a}
J.~J.~H. Miller and S.~Wang.
\newblock A new non-conforming petrov-galerkin finite-element method with
  triangular elements for a singularly perturbed advection-diffusion problem.
\newblock {\em IMA J.\ Numer.\ Anal.}, 14(2):257--276, 1994.

\bibitem{RiordanE1991a}
Eugene O'Riordan and Martin Stynes.
\newblock A globally uniformly convergent finite element method for a
  singularly perturbed elliptic problem in two dimensions.
\newblock {\em Math.\ Comp.}, 57:47--62, 1991.

\bibitem{PinnauR2004a}
Ren\'e Pinnau.
\newblock Uniform convergence of an exponentially fitted scheme for the quantum
  drift diffusion model.
\newblock {\em SIAM J.\ Numer.\ Anal.}, 42:1648--1668, 2004.

\bibitem{RaviartThomasElements}
P.A. Raviart and J.M. Thomas.
\newblock A mixed finite element method for second order elliptic problems.
\newblock In I.~Galligani and E.~Magenes, editors, {\em Lecture notes in
  Mathematics, Vol. 606}. Springer-Verlag, 1977.

\bibitem{SaccoR1995a}
R.~Sacco, E.~Gatti, and L.~Gotusso.
\newblock The patch test as a validation of a new finite element for the
  solution of convection-diffusion equations.
\newblock {\em Comput.\ Methods Appl.\ Mech.\ Eng.}, 124:113--124, 1995.

\bibitem{SaccoR1997a}
R.~Sacco and F.~Saleri.
\newblock Stabilized mixed finite volume methods for convection-diffusion
  problems.
\newblock {\em East West J. Numer. Math.}, 5:291--311, 1997.

\bibitem{SaccoR1998a}
R.~Sacco and M.~Stynes.
\newblock Finite element methods for convection-diffusion problems using
  exponential splines on triangles.
\newblock {\em Computers Math. Applic.}, 35(3):35 -- 45, 1998.

\bibitem{SaccoR1999a}
Riccardo Sacco, Emilio Gatti, and Laura Gotusso.
\newblock A nonconforming exponentially fitted finite element method for
  two-dimensional drift-diffusion models in semiconductors.
\newblock {\em Numerical Methods for Partial Differential Equations},
  15(2):133--150, 1999.

\bibitem{ScharfetterD1969}
D.~L. Scharfetter and H.~K. Gummel.
\newblock Large-signal analysis of a silicon read diode oscillator.
\newblock {\em IEEE T.\ Electron Dev.}, 16:64--77, 1969.

\bibitem{ScheichlR2002}
R.~Scheichl.
\newblock Decoupling three-dimensional mixed problems using divergence-free
  finite elements.
\newblock {\em SIAM J.\ Sci.\ Comput.}, 23(5):1752--1776, 2002.

\bibitem{WangJ2009a}
J.~Wang, Y.~Wang, and X.~Ye.
\newblock A robust numerical method for stokes equations based on
  divergence-free h(div) finite element methods.
\newblock {\em SIAM J.\ Sci.\ Comput.}, 31(4):2784--2802, 2009.

\bibitem{WangJ2007a}
J.~Wang and X.~Ye.
\newblock New finite element methods in computational fluid dynamics by
  {H(div)} elements.
\newblock {\em SIAM J.\ Numer.\ Anal.}, 45:1269--1286, 2007.

\bibitem{ZhouY2012b}
Y.~C. Zhou.
\newblock Electrodiffusion of lipids on membrane surfaces.
\newblock {\em J. Chem. Phys.}, 136:205103, 2012.

\end{thebibliography}
\end{document}